%% file: main.tex
\documentclass[11pt]{article}
\input{Preamb.tex}
\usepackage{graphicx}
\usepackage{subcaption}

\newcommand{\ProDes}{$\mathtt{ProxBundle}$}

\newcommand{\R}{\mathbb{R}}

\newcommand{\dom}{\operatorname{dom}}

\newcommand{\prox}{\operatorname{prox}}

\usepackage{array}
\usepackage{booktabs}

\usepackage[numbers,merge,sort&compress]{natbib}
\begin{document}

\title{
 \bf \Large Revisiting Proximal Bundle Methods: Improved Rates under H{\"o}lder Smoothness \thanks{This work is supported by NSF CAREER 2340713 and CMMI 2320697. Emails: \{fliao,zhengy\}@ucsd.edu.} 
}
\author[1]{Feng-Yi Liao}
\author[1]{Yang Zheng}
\affil[1]{\small Department of Electrical and Computer Engineering, University of California San Diego}
\date{\small \today \vspace{-5ex}} 

\maketitle

\begin{abstract}
Proximal bundle methods (PBMs) are classical algorithms for nonsmooth convex optimization. Existing analyses of the classical PBM couple the null steps with the descent test. This coupling obscures how the bundle updates approximate the proximal subproblem. In~this work, we consider composite objectives $F=f+h$ and view each null-step cycle as an inner bundle subroutine, called \ProDes. We analyze \ProDes{} independently of any stopping criterion under general bundle model conditions and show that it automatically adapts to H{\"o}lder smoothness. Combining this inner-loop analysis with the descent-step analysis yields sharper complexity bounds for the classical PBM. For any fixed proximal parameter, its overall complexity is $\mathcal O\big(\epsilon^{-\frac{3-\nu}{1+\nu}}\big)$ for $\nu\in[0,1)$ and $\mathcal O(\epsilon^{-1})$ for $\nu=1$, where $\nu$ is the H{\"o}lder smoothness~exponent. Choosing the proximal parameter proportional to $\epsilon$ improves the rate to $\mathcal O\big(\epsilon^{-\frac{2}{1+\nu}}\big)$ for $\nu \in [0,1)$. These are the first guarantees under H\"older smoothness with $\nu\in(0,1)$ for the classical descent test. We further introduce an absolute model-error test. The resulting PBM variant admits a clean inexact proximal-point analysis, and for any fixed proximal parameter, achieves the same complexity $\mathcal O\big(\epsilon^{-\frac{2}{1+\nu}}\big)$. Overall, our analysis separates the roles of the descent and null steps and gives a modular understanding of PBMs across different~tests.
\end{abstract}

\section{Introduction}

In this work, we consider the composite convex optimization problem 
\begin{equation}\label{eq:pb-main}
\min_{x\in\R^n}F(x),
 \qquad
 F(x):=f(x)+h(x),
\end{equation}
where \(f:\R^n\to\R\) is finite-valued and convex, and \(h:\R^n\to\R\cup\{+\infty\}\) is proper, lower semicontinuous, and convex. Throughout, we assume $F$ attains its minimum value $F^\star$ with the nonempty solution set
$
 X^\star:= \{x \in \R^n\mid F(x)=F^\star\}. 
$ 
This composite form in \cref{eq:pb-main} provides substantial flexibility in modeling. In particular, the function $h$ may encode a nonsmooth regularizer or a closed convex constraint through an indicator function, and the function $f$ is available through a first-order oracle and may be nonsmooth. This form covers a wide range of problems in data science and engineering. For instance, $\ell_1$ regularization is standard in sparse estimation~\cite{tibshirani1996regression}, and the Lagrangian dual of a conic program can be nonsmooth~\cite{nesterov1994interior}.

Proximal bundle methods are classical algorithms to solve nonsmooth convex problems, introduced in the 1970s \cite{Lemarechal1975,mifflin1977algorithm}. Conceptually, they are similar to modern model-based methods \cite{davis2019stochastic,drusvyatskiy2021nonsmooth}. In model-based methods, each iteration builds a tractable approximation of the original function, known as model $F_k$, and then updates the iterate by applying a proximal step:
\begin{equation*}
    x_{k+1} = \argmin F_k(x) + \frac{\rho}{2}\|x - x_k\|^2,
\end{equation*}
where $\rho > 0$ is a proximal parameter. Unlike the usual model-based schemes, one unique feature of classical bundle methods is that they only update the iterate $x_{k+1}$ when the objective decrease in $F$ is at least a $\beta$ fraction of the decrease the model predicted. Moreover, bundle methods retain first-order information from past iterations so that $F_k$ reflects more than the geometry of $F$ near~$x_k$.

These modifications appear minor but have a striking consequence: the iterates from the classical bundle methods with any fixed parameter $\rho>0$ and $\beta\in (0,1)$ converge to a minimizer of $F$ when it exists~\cite{kiwiel1990proximity,kiwiel1995approximations}, \cite[Theorem 7.16]{ruszczynski2011nonlinear}. Moreover, the convergence rate naturally adapts to the regularity of the problem with little algorithm modification. The same proximal bundle algorithm converges at $O(\epsilon^{-3})$ for Lipschitz convex problems \cite{kiwiel2000efficiency} and at $O(\epsilon^{-1})$ for smooth problems. The rate further improves automatically when the problem enjoys the H\"older growth property~\cite{diaz2023optimal}. The nonsmooth rate also improves to the optimal $O(\epsilon^{-2})$ once $\rho$ is scaled with the target accuracy~\cite{diaz2023optimal}. This automatic adaptivity is in sharp contrast with other first-order methods, whose stepsizes must be tuned carefully based on the problem regularity. For example, gradient descent and its accelerated variants require a stepsize on the order of the inverse smoothness constant \cite[Chapter 2]{nesterov2018lectures}, subgradient methods require a carefully controlled diminishing sequence \cite[Chapter 3.2]{nesterov2018lectures}, and the universal method covers both smooth and nonsmooth cases by running a line search at every iteration~\cite{nesterov2015universal}. 
This raises a basic question: \emph{where does the adaptivity~come~from?} 

The asymptotic convergence of the classical bundle methods is well established under various assumptions \cite{correa1993convergence,lemarechal1981bundle,lemarechal1994condensed,kiwiel1990proximity,kiwiel1990proximity,kiwiel1995approximations,ruszczynski2011nonlinear}, and there are also successful applications \cite{de2014convex,helmberg2000spectral,liao2026overview,ding2023revisiting,apkarian2008trust,watanabe2026weak}. Nonasymptotic~guarantees are considerably more involved to establish. Classical nonasymptotic analysis was first initiated in the early 2000s \cite{kiwiel2000efficiency}, and a comprehensive treatment was only given recently in \cite{diaz2023optimal}~whose scope is restricted to the unconstrained setting $h\equiv0$. Establishing such guarantees requires~tracking two distinct types of iterations. The standard analysis~\cite{kiwiel2000efficiency,diaz2023optimal} adopts a nested structure,~which has a sequence of \emph{descent steps}, at which the proximal center moves, and, between consecutive descent steps, a cycle of \emph{null steps}, at which the center is fixed and the model is refined. The analysis of the sequence of descent steps is transparent. Given a center $\bar x$, write $\Delta := F(\bar x) - \min_x\{F(x) + \frac{\rho}{2}\|x-\bar x\|^2\}$ for the proximal gap at $\bar x$. Every descent step decreases the objective by at least $\beta\Delta$, which is the decrease guaranteed by an exact proximal point step, discounted by $\beta$. The convergence of the descent steps thus inherits that of the proximal point method (PPM). The lower bounds on $\Delta$ can be sharpened under the H\"older growth of $F$, which drive the descent-step adaptivity established~in~\cite{diaz2023optimal}. 

The cycle of null steps is far less transparent. The argument of~\cite{diaz2023optimal} tracks the approximated proximal gap $\widetilde\Delta_k := F(\bar x) - \min_x\{F_k(x) + \frac{\rho}{2}\|x-\bar x\|^2\}$, and shows that as long as null steps continue to occur, it obeys the following recursion,
\begin{align}
    \label{eq:recursion-classical}
   \widetilde \Delta_{k+1} \leq \widetilde\Delta_{k} - c_0 \times \widetilde \Delta_{k}^2, 
\end{align}
where $c_0$ is a constant. Since $F_k \le F$ implies $\widetilde\Delta_k \ge \Delta$, the cycle of null steps must terminate, and the recursion \cref{eq:recursion-classical} converts this into a bound in terms of $\Delta$. This recursion, however, is not a statement about the bundle iterations alone. It is derived from the \emph{failure} of the classical descent~test
\begin{align}
    \label{eq:PMB-test-intro}
    \beta (F(\bar x) - {F}_k(z_{k+1}) ) \leq F(\bar x) - F(z_{k+1}), 
\end{align}
where $z_{k+1} = \arg \min_x \{ F_k(x) + \frac{\rho}{2}\|x-\bar x\|^2\}$ is the candidate point, and $\beta \in (0,1)$. Moreover, the constant $c_0$ in \cref{eq:recursion-classical} carries a factor of $(1-\beta)^2$ \cite{diaz2023optimal}. The convergence of null steps is thus not defined independently of the rule \cref{eq:PMB-test-intro}. Consequently, the existing analyses~\cite{diaz2023optimal} treat regularity in two isolated cases: $c_0$ is a constant when $f$ is Lipschitz, whereas a separate argument bounds the number of null steps uniformly when $f$ is $L_1$-smooth, with no results connecting the two cases.

In this work, we show that the null-step cycle adapts to H\"older smoothness on its own, independently of the descent test \cref{eq:PMB-test-intro}. The two forms of adaptivity then have separate and clear sources: adaptivity to H\"older growth comes from the PPM structure of the descent steps, and adaptivity to H\"older smoothness is intrinsic to the bundle iterations.

\subsection{Our contributions} 

\begin{table*}[t]
\centering
\setlength{\belowcaptionskip}{0pt}
\caption{
Complexity of PBMs for finding an iterate $x_k$ satisfying $F(x_k)-\min_x F(x)\leq \epsilon$, where $F=f+h$ and $f$ satisfies \cref{eq:weakly-smooth-intro}. Here, we present only the leading term with $\epsilon$ sufficiently small. Let $X^\star$ be the optimal solution set. For~the classical PBM, we denote  $ \beta \in (0,1),$ $ r_{\nu} = \frac{1-\nu}{1+\nu}$, and $ \sup \{\Dist(x,X^\star) \mid F(x) \leq F(y_0)\} \leq D < +\infty$. For the PBM variant, we denote $d_0 = \Dist(y_0, X^\star)$ and $
A=
\frac{
\log(1+L_1/\rho)+\log(3+2R_0)
}{
\log(1+\rho/L_1)
}$ with $R_0=\frac{F(y_0) - F^\star}{\rho d_0^2}$.
}
\label{tab:intro-rate}

\small
\renewcommand{\arraystretch}{1.35}
\setlength{\tabcolsep}{4pt}

\begin{tabular}{@{}lccccc@{}}
\toprule
Method
&
$\rho$
&
\shortstack{Null-step\\test}
&
$\nu=0$
&
$\nu\in(0,1)$
&
$\nu=1$
\\
\midrule

\shortstack{Classical\\PBM}
&
$\rho>0$
&
\cref{eq:PMB-test-intro}
&
$\displaystyle
\mathcal{O}\!\left(
\frac{\rho L_0^2D^4}
{\beta(1-\beta)\epsilon^3}
\right)$
&
$\displaystyle
\mathcal{O}\!\left(
\frac{
L_\nu^{\frac{2}{1+\nu}}
\rho^{r_\nu}
D^{\frac{4}{1+\nu}}
}{
\beta(1-\beta)^{r_\nu}
\epsilon^{\frac{3-\nu}{1+\nu}}
}
\right)$
&
$\displaystyle
\mathcal{O}\!\left(
\frac{\rho D^2}{\beta\epsilon}
\left[
1+
\frac{
\log_+\!\left(
\frac{L_1^2}{(1-\beta)\rho^2}
\right)
}{
\log(1+\rho/L_1)
}
\right]
\right)$
\\[1.3ex]

\shortstack{Classical\\PBM}
&
$\displaystyle \rho=\frac{\epsilon}{D^2}$
&
\cref{eq:PMB-test-intro}
&
$\displaystyle
\mathcal{O}\!\left(
\frac{L_0^2D^2}
{\beta(1-\beta)\epsilon^2}
\right)$
&
$\displaystyle
\mathcal{O}\!\left(
\frac{
L_\nu^{\frac{2}{1+\nu}}D^2
}{
\beta(1-\beta)^{r_\nu}
\epsilon^{\frac{2}{1+\nu}}
}
\right)$
&
$\displaystyle
\widetilde{\mathcal{O}}\!\left(
\frac{L_1D^2}{\beta\epsilon}
\left[
1+
\log_+\!\left(
\frac{L_1^2D^4}
{(1-\beta)\epsilon^2}
\right)
\right]
\right)$
\\[1.3ex]

\shortstack{PBM\\variant}
&
$\rho>0$
&
\shortstack{
\cref{eq:new-test}
}
&
$\displaystyle
\mathcal{O}\!\left(
\frac{L_0^2d_0^2}{\epsilon^2}
\right)$
&
$\displaystyle
\mathcal{O}\!\left(
\frac{
L_\nu^{\frac{2}{1+\nu}}d_0^2
}{\epsilon^{\frac{2}{1+\nu}}
}
\right)$
&
$\displaystyle
\mathcal{O}\!\left(
\frac{\rho d_0^2}{\epsilon}
(1+A)
\right)$
\\

\bottomrule
\end{tabular}

\end{table*}

We revisit the classical PBM and establish its convergence rates under H\"older smoothness; see \Cref{tab:intro-rate}. In particular, we consider the setting where the function $f$ in \cref{eq:pb-main} is H{\"o}lder smooth,~i.e., 
\begin{align}
    \label{eq:weakly-smooth-intro}
    \|\nabla f(x) - \nabla f(y)\| \leq L_{\nu}\|x - y\|^{\nu}, \quad \forall x, y \in \RR^n, 
\end{align}
where $L_{\nu} \geq 0 $ denotes the H{\"o}lder-smoothness constant, and $\nu \in [0,1] $ is fixed. With a slight abuse of notation, if $f$ is not differentiable,  $\nabla f(x)$ in \cref{eq:weakly-smooth-intro} denotes any subgradient of $\partial f(x)$. When $\nu = 0$, \cref{eq:weakly-smooth-intro} recovers the case with bounded subgradients. An $L$-Lipschitz function satisfies \cref{eq:weakly-smooth-intro} with $\nu = 0$ and $L_{0} = 2L$. If $\nu = 1$, \cref{eq:weakly-smooth-intro} recovers the class of $L_{1}$-smooth functions. Our main results are~as~follows.

We first view the null-step cycle as a standalone subroutine, called \texttt{ProxBundle}, that approximately solves the proximal subproblem $\min_x \{F(x) + \frac{\rho}{2}\|x - \bar x\|^2\}$ at a fixed center $\bar x$. This naturally leads to a double-loop interpretation of the classical PBM.  We analyze the inner loop, \texttt{ProxBundle}, without specifying any stopping criterion. For convenience, we use a separate index $j$ to denote the inner bundle iterations. 
Let $F_j$ denote the bundle model at the inner bundle iteration $j$ and 
$
  \delta_j := \min_x\Big\{F(x) + \tfrac{\rho}{2}\|x-\bar x\|^2\Big\}
            - \min_x\Big\{F_j(x) + \tfrac{\rho}{2}\|x-\bar x\|^2\Big\}
            \;\ge\;0 
$ 
denote the gap between the~true and approximated proximal values. If $f$ is H\"older smooth, i.e., satisfying \cref{eq:weakly-smooth-intro}, then we have 
\begin{equation}\label{eq:contribution-recursion}
  \delta_{j+1} \;\le\; \delta_j - c_1\times\delta_{j+1}^{\frac{2}{1+\nu}},
  \qquad \forall j \ge 0,
\end{equation}
where $c_1>0$ depends only on $\rho$, $L_\nu$, and $\nu$. Thus, $\delta_j$ converges to zero with the rate $O(j^{-\frac{1+\nu}{1-\nu}})$ for $\nu\in[0,1)$ and linearly for $\nu=1$. In contrast to \cref{eq:recursion-classical}, the recursion \cref{eq:contribution-recursion} involves neither the descent parameter $\beta$ nor the test \cref{eq:PMB-test-intro}. In other words,  \texttt{ProxBundle} solves the proximal subproblem at a rate that adapts to the smoothness of $f$ on its own. Note that \cref{eq:contribution-recursion} only requires the standard bundle conditions in \cref{assump:bm-single}, and therefore covers the two-cut, multi-cut, and spectral bundle models. Our analysis also applies to the \emph{computable} model error $e_j := F(z_{j+1}) - F_j(z_{j+1})$, which upper bounds $\delta_j$ and inherits its rate over a dyadic window,
$
  \min_{t+1\le j\le 2t} e_j
  \;\le\; O\big(t^{-\frac{1+\nu}{1-\nu}}\big), \, \forall t\ge1. 
$ 

We then analyze the complexity of the classical PBM by combining the inner-loop rate above with the PPM convergence of the descent steps. For any fixed $\rho>0$ and $\beta \in (0,1)$, the classical PBM finds an iterate satisfying $F(x) - F^\star \le \epsilon$ within $O\big(\epsilon^{-1}\big)$ iterations when $\nu = 1$, and $O\big(\epsilon^{-\frac{3-\nu}{1+\nu}}\big)$ iterations when $\nu \in [0,1)$. The convergence rate improves to $O(\epsilon^{-\frac{2}{1+\nu}})$ when $\rho$ is chosen proportional to $\epsilon$ for $\nu \in [0,1)$. These correspond to the first two rows in \Cref{tab:intro-rate}. To the best of our knowledge, these are the first non-asymptotic guarantees for the classical PBM with the descent test \cref{eq:PMB-test-intro} under H\"older smoothness; the regime $\nu\in(0,1)$ was previously not discussed. In the two extreme cases $\nu\in\{0,1\}$, our bounds recover the known rates in \cite{diaz2023optimal} while improving their dependence on the algorithmic parameters. In particular, the factor $(L_1/\rho)^2$ in~\cite{diaz2023optimal} is replaced by $ \log(L_1^2/\rho^2)$ in the smooth case when $\rho \ll L_{1}$,  and the dependence on $(1-\beta)^{-1}$ improves from quadratic to linear in the nonsmooth case; see \cref{tab:comparison-classical-PBM} in \Cref{subsec:comparison-state-of-art} for a detailed comparison. Our results also hold for the composite problem~\cref{eq:pb-main}, whereas~\cite{diaz2023optimal} is restricted to the unconstrained setting with $h\equiv0$.

Finally, the computable model error $e_j$ suggests a way to terminate \texttt{ProxBundle} other than~\cref{eq:PMB-test-intro}. Each candidate produced by the inner loop satisfies the inexact proximal inclusion
$
  z_{j+1} = \bar x - \tfrac{1}{\rho}s_{j+1},
  \, s_{j+1} \in \partial_{e_j} F(z_{j+1}),
$
where $\partial_{e}F$ denotes the $e$-inexact subdifferential of $F$. Since $e_j$ is directly computable, we may terminate the inner loop once
\begin{equation}\label{eq:new-test}
    e_j \le \frac{\epsilon}{2},
\end{equation}
where $\epsilon$ is the target accuracy. This gives an inexact PPM step with an explicitly controlled error. The resulting PBM variant achieves $O(\epsilon^{-\frac{2}{1+\nu}})$ total complexity for \emph{any} fixed $\rho>0$. Note that the classical PBM attains the same rate only by tuning $\rho = \Theta(\epsilon)$. The analysis of this PBM variant is also modular and transparent: the outer complexity follows from standard inexact PPM arguments and the inner complexity directly comes from the subroutine analysis above.

We note that the rates in \Cref{tab:intro-rate} are not entirely new on bundle-type methods in the literature, and different variants of the stopping rule \cref{eq:new-test} have been used in modern PBM developments \cite{liang2021proximal,liang2024unified,liang2026proximal,liang2025primal,
fersztand2024modified}. In particular, comparable rates $O(\epsilon^{-\frac{2}{1+\nu}}\log(1/\epsilon))$ have been established for modern PBM variants with different stopping rules~\cite{liang2024unified,liang2026proximal}, which we improve on by removing the extra logarithmic term and thus match the universal primal gradient method of~\cite{nesterov2015universal}. We provide a further comparison with modern PBM variants \cite{liang2021proximal,liang2024unified,liang2026proximal,liang2025primal,fersztand2024modified} in \Cref{subsection:comparison-modern-PBM}. To our knowledge, the rates for the classical PBM with the descent test  \cref{eq:PMB-test-intro} under H\"older smoothness are new. A conceptual contribution is that we analyze the convergence of \ProDes{} under a general model class with no stopping criterion, which allows a unified analysis for both classical PBM and modern PBM variants. 

\subsection{Related work}

Early PBM developments primarily established asymptotic convergence under various assumptions \cite{lemarechal1981bundle,lemarechal1994condensed,kiwiel1990proximity,kiwiel1995approximations,ruszczynski2011nonlinear}. The first non-asymptotic guarantee for the classical PBM is due to Kiwiel~\cite{kiwiel2000efficiency}, who established an $O(\epsilon^{-3})$ complexity for Lipschitz convex objectives. An improved rate of $\bigO(\epsilon^{-1}\log(1/\epsilon))$ was established in \cite{du2017rate} for convex problems satisfying quadratic growth (including strong convexity as a special case \cite{liao2024error}). A comprehensive understanding of the classical PBM appeared only~recently in the work of D\'iaz and Grimmer~\cite{diaz2023optimal}, which shows that the method adapts to general~H\"older growth beyond strong convexity and that a suitably tuned proximal parameter recovers the optimal $O(\epsilon^{-2})$ rate for nonsmooth problems. Their analysis considers only the unconstrained setting $h\equiv0$ and treats the smooth and Lipschitz cases separately, leaving the intermediate H\"older-smooth case open. Our analysis shows that classical bundle iterations naturally adapt to H\"older smoothness.

Although the PBM is motivated by the PPM, the precise sense in which the bundle iterations solve the proximal subproblem remains unclear. As discussed above, the analysis in \cite{diaz2023optimal} interweaves the descent and null steps. This issue was partially addressed in the work of Liang and~Monteiro~\cite{liang2021proximal}, which showed that a cycle of null steps can be interpreted as approximately solving the true proximal subproblem. In particular, \cite{liang2021proximal} replaces the classical descent test with a criterion based on the proximal subproblem value and establishes an $O(\epsilon^{-2})$ complexity for nonsmooth problems. A subsequent unified analysis~\cite{liang2024unified} covers a hybrid regularity condition interpolating between the bounded-subgradient and $L$-smooth settings, and derives $\mathcal{O}(\epsilon^{-\frac{2}{1+\nu}}\log(1/\epsilon))$ for H\"older-smooth objectives. Closest to our work, Liang and Chen~\cite{liang2026proximal} analyze a bundle subroutine directly under H\"older smoothness and obtain $\mathcal{O}(\epsilon^{-\frac{2}{1+\nu}}\log(1/\epsilon))$. Their analysis is restricted to the full cutting-plane model and incurs an additional logarithmic factor. In this work, we explicitly analyze the bundle iterations, called \ProDes{}, which solve the true proximal subproblem with a non-asymptotic rate adapted to H\"older smoothness. Our analysis accommodates the general bundle model in \cref{assump:bm-single} and removes the extra logarithmic factor in \cite{liang2026proximal,liang2024unified}.

Beyond convergence and complexity analysis, numerous variants of PBM have been developed for different problem classes and applications. These include the bundle-level method \cite{lemarechal1995new}, trust-region bundle methods \cite{astorino2011piecewise}, proximal bundle methods for nonconvex problems \cite{liao2025proximal,hare2009computing,hare2010redistributed,liang2023proximal}, spectral bundle methods for semidefinite programming \cite{ding2023revisiting,helmberg2000spectral,liao2026overview}, stochastic proximal bundle methods \cite{liang2024single,liang2025multi}, bundle methods for problems with a $\mathcal{V}\mathcal{U}$-structure \cite{mifflin2012science,mifflin2005vu}, and inexact bundle methods \cite{de2020bundle,hare2016proximal}. When applied to dual problems, PBM also exhibits a close primal-dual relationship with augmented Lagrangian methods \cite{lemarechal2001lagrangian,liao2025bundle}. Related developments include a duality between the cutting-plane scheme and the conditional
gradient method~\cite{liang2025primal}, and an interpretation of null steps as Frank-Wolfe iterations~\cite{fersztand2024modified}. Other recent developments include adaptive or parameter-free schemes~\cite{guigues2026universal,monteiro2024parameter,guigues2026complexity}. 

Among these variants, the bundle-level method has been combined with acceleration techniques to obtain universally optimal methods for H{\"o}lder smooth functions \cite{lan2015bundle}. In particular, the method of \cite{lan2015bundle} achieves the optimal $\bigO(\epsilon^{-1/2})$ complexity in the smooth convex setting. Our recent work \cite{liao2025accelerated} showed that the classical PBM can also attain the optimal $\bigO(\epsilon^{-1/2})$ rate when integrated with the Nesterov-type acceleration. Whether the classical PBM admits a universally accelerated rate of $O(\epsilon^{-\frac{2}{1+3\nu}})$ as in the universal fast gradient method \cite{nesterov2015universal} remains open, and we expect it to require an acceleration mechanism beyond the standard bundle update.

\subsection{Paper Outline}
\Cref{section:prelimianries} introduces necessary background on the proximal point method and the classical proximal bundle method. In \Cref{sec:bundle-terations}, we interpret the classical PBM as a double-loop algorithm, analyze the inner bundle iterations, and establish their convergence under H{\"o}lder smoothness. \Cref{sec:revisit-classical-PBM} revisits the classical PBM and establishes convergence guarantees for the composite problem \cref{eq:pb-main}. In \Cref{sec:PBM-variant}, we introduce a new PBM variant based on the absolute error test \cref{eq:new-test} and also compare it with modern PBM designs \cite{liang2021proximal,liang2024unified,liang2026proximal,liang2025primal,
fersztand2024modified}. We present numerical experiments in \cref{sec:numerics}. Finally, \cref{sec:conclusion} concludes the paper. Some additional proof details are presented in the appendix.     

\section{Preliminaries} \label{section:prelimianries}

\subsection{The proximal point method} \vspace{-1mm}
Consider the convex optimization problem \cref{eq:pb-main}. Given an initial point $x_0 \in \dom F$,  the proximal point method (PPM) generates a sequence of iterates as follows 
\begin{equation} \label{eq:PPM-iterate}
    x_{k+1} = \prox_{F/\rho}(x_k)
    := \argmin_{x\in\R^n}
    \left\{
        F(x)+\frac{\rho}{2}\|x-x_k\|^2
    \right\},
    \qquad k=0,1,2,\ldots,
\end{equation}
where $\rho > 0$ is a proximal parameter. The quadratic term in \cref{eq:PPM-iterate} makes each proximal subproblem strongly convex, which admits a unique minimizer. The PPM iterates are thus
well-defined. The convergence of the PPM for nonsmooth convex optimization has been studied since the 1970s \cite{rockafellar1976monotone}, and its iterates satisfy \(F(x_k)-F^\star=\mathcal{O}(1/k)\); see, for example, \cite[Theorem~2.1]{guler1991convergence}. Although the proximal updates are generally not directly tractable, the PPM provides a conceptual foundation
for many proximal and bundle-type algorithms
\cite{drusvyatskiy2017proximal,diaz2023optimal,liang2021proximal,
liao2025bundle,liao2025proximal}.

In practice, one often computes an approximate proximal point $
  x_{k+1}
    \approx
    \prox_{F/\rho}(x_k)
 $. 
A classical inexactness criterion \cite{rockafellar1976monotone} requires \vspace{-3mm}
\begin{equation} \label{eq:Rockafellar-1}
    \|x_{k+1}-\prox_{F/\rho}(x_k)\|
    \leq \epsilon_k,
    \,
    \sum_{k=0}^{\infty}\epsilon_k<\infty.
\end{equation}
This summable error condition ensures convergence of the inexact PPM whenever a solution exists \cite{rockafellar1976monotone}. Another common criterion requires an approximate solution in objective value:
\begin{equation}
\label{eq:inexact-func}
    F(x_{k+1})
    +\frac{\rho}{2}\|x_{k+1}-x_k\|^2
    \leq
    \min_{x\in\R^n}
    \left\{
        F(x)+\frac{\rho}{2}\|x-x_k\|^2
    \right\}
    +\eta_k,
\end{equation}
where \(\eta_k\geq0\) is a suitable error; see, e.g., \cite{salzo2012inexact,lin2018catalyst}. Note that both criteria involve quantities that are generally unavailable: the exact proximal point in \cref{eq:Rockafellar-1} and the exact proximal value~in~\cref{eq:inexact-func}.

Proximal bundle methods exploit the composite structure \(F=f+h\), where \(h\) is treated exactly in the proximal update and \(f\) may be nonsmooth and difficult to minimize directly. They~approximate \(f\) with a tractable lower model \(f_k\), and solve the resulting proximal subproblem with \(F_k:=f_k+h\). The model error \(F(z)-F_k(z)=f(z)-f_k(z)\) at a candidate point \(z\) provides a computable measure of inexactness, offering an implementable
alternative to the criteria \cref{eq:Rockafellar-1,eq:inexact-func}. 

\vspace{-1mm}
\subsection{Classical proximal bundle methods}
\label{subsec:PBM}
 We here formally introduce the family of proximal bundle methods to solve the composite convex problem \cref{eq:pb-main}. The standard analyses most closely related to ours \cite{diaz2023optimal} are developed~primarily for the unconstrained setting, corresponding to \(h\equiv 0\). We slightly extend the setting in \cite{diaz2023optimal,kiwiel2000efficiency} and consider the general composite case \cref{eq:pb-main}. Similar composite formulations have also been used in modern proximal bundle methods; see, for example, \cite{liang2021proximal,liang2023proximal,liang2024unified}. 

Conventionally, classical PBMs are presented as single-loop schemes with a dynamic descent test. At the $k$-th iteration, the method constructs a convex lower model $f_k: \mathbb{R}^n \to \mathbb{R}$ of $f$ and computes a candidate point $z_{k+1}$ by solving the proximal problem 
\begin{align}
    \label{eq:pbm-subproblem-single}
    z_{k+1} = \arg \min_{x} \;f_k(x)+ h(x) +\frac{\rho}{2}\| x- y_k\|^2,
\end{align}
where \(y_k\) is the current prox center and \(\rho>0\) is the proximal parameter. Unlike other model-based algorithms, the bundle method does not immediately move the next iterate to $z_{k+1}$. Instead, it tests whether this candidate $z_{k+1}$  makes sufficient progress. In particular, the classical PBM checks the descent condition \cref{eq:PMB-test-intro}  with the current center $\bar x=y_k$. Thus, the test becomes
$
F(y_k)-F(z_{k+1})
\geq
\beta\bigl(F(y_k)-F_k(z_{k+1})\bigr).
$  
The test \cref{eq:PMB-test-intro} means that the candidate achieves at least $\beta$-fraction of the decrease predicted by the model. If \cref{eq:PMB-test-intro} holds, the candidate is accepted, and the proximal center is updated according to \(y_{k+1}=z_{k+1}\); this is called a \emph{serious step} (or \emph{descent step}). Otherwise, the candidate is rejected as a new center and \(y_{k+1}=y_k\); this is called a \emph{null step}. Although a null step leaves the prox center unchanged, the new function and subgradient information obtained at \(z_{k+1}\) is used to refine the bundle model  $f_{k+1}$. The classical PBM is listed in \Cref{alg:bundle}. 

\begin{algorithm}[t]
\caption{Classical Proximal Bundle method (PBM)}\label{alg:bundle}
\begin{algorithmic}[1]
\Require $y_0 \in \dom F, \rho > 0, \beta \in (0,1),  k_{\max}\in \mathbb{N}$
\State Set $z_0 = y_0$.
\For{$k=0,1,2, \ldots,k_{\max} -1 $}
    \State Construct the bundle model $f_{k}$ following \cref{assump:bm-single}.
    \State Compute the candidate point $z_{k+1}$ by \cref{eq:pbm-subproblem-single}. 
    \If{\cref{eq:PMB-test-intro} is satisfied}
        \State Set $y_{k+1} = z_{k+1}$, \hfill \Comment{\textit{descent step}}
    \Else
        \State Set $y_{k+1} = y_k$.  \hfill \Comment{\textit{null step}}
    \EndIf
\EndFor
\end{algorithmic}
\end{algorithm}

To state the bundle model assumption, we note that the first-order optimality for \cref{eq:pbm-subproblem-single} gives 
\begin{equation}
0 \in \partial f_k(z_{k+1})+ \partial h(z_{k+1}) +{\rho}(z_{k+1}- y_k). 
\end{equation}
Here, for a proper convex function
\(\varphi:\R^n\to\R\cup\{+\infty\}\), its subdifferential at
\(x\in\dom\varphi\) is
$
    \partial\varphi(x)
    :=
    \left\{
        g\in\R^n:
        \varphi(u)\geq
        \varphi(x)+\innerproduct{g}{u-x}
        \ \text{for all }u\in\R^n
    \right\}.
$ 
Its inexact version with inexactness $\epsilon \geq 0$ is defined as $
    \partial_{\epsilon}\varphi(x)
    :=
    \left\{
        g\in\R^n:
        \varphi(u)\geq
        \varphi(x)+\innerproduct{g}{u-x} - \epsilon
        \ \text{for all }u\in\R^n
    \right\}.
$ 
Consequently, there exist
\(a_{k+1}\in\partial f_k(z_{k+1})\) and
\(b_{k+1}\in\partial h(z_{k+1})\) such that $
    a_{k+1}+b_{k+1}
    =\rho(y_k-z_{k+1}). $ 
This motivates the aggregate condition below.

\begin{assumption}
\label{assump:bm-single}
For every \(k\geq 0\), the convex model \(f_k:\R^n\to\R\) satisfies the following conditions:
\begin{enumerate}[leftmargin=*]
    \setlength{\itemsep}{0pt}
    \item \textbf{Global Minorant.} The model is a lower approximation of \(f\):
    $
        f_k(x)\leq f(x),
        \, \forall x\in\R^n.
    $

    \item \textbf{New subgradient cut.} For some \(g_k\in\partial f(z_k)\), we have
    \begin{equation}
        \label{eq:bundle-subgradient}
        f_k(x)
        \geq
        f(z_k)+\innerproduct{g_k}{x-z_k},
        \qquad \forall x\in\R^n.
    \end{equation}

    \item \textbf{Aggregation cut.} If $k \geq 1$ and iteration \(k-1\) is a null step, then there exist
    $
        a_k\in\partial f_{k-1}(z_k)$, $b_k\in\partial h(z_k)
    $ 
    and $a_k+b_k =\rho(y_k-z_k)$ such that 
    \begin{align}
        f_k(x)
        &\geq
        f_{k-1}(z_k)+\innerproduct{a_k}{x-z_k},
        \qquad \forall x\in\R^n.       \label{eq:bundle-aggregation} 
    \end{align}
\end{enumerate}
\end{assumption}

Recall that \(F_k:=f_k+h\). The first requirement ensures that the quantity $F(y_k) - F_k(z_{k+1})$ in \cref{eq:PMB-test-intro} is always nonnegative and every serious step guarantees \(F(z_{k+1})\leq F(y_k)\). Condition \cref{eq:bundle-subgradient} incorporates the first-order information of the original $f$ at $z_{k}$, and condition \cref{eq:bundle-aggregation} retains suitable first-order information of the previous model $f_{k-1}$  at $z_{k}$. The aggregation  \cref{eq:bundle-aggregation}  also ensures that the proximal value
increases during a null step. Indeed, since \(y_k=y_{k-1}\), we have
\begin{align}
    F_{k-1}(z_k)+\frac{\rho}{2}\|z_k-y_k\|^2  &=
    \min_x\left\{
        F_{k-1}(z_k)
        +\innerproduct{a_k+b_k}{x-z_k}
        +\frac{\rho}{2}\|x-y_k\|^2
    \right\} \nonumber\\
    &\leq
    \min_x\left\{
        F_k(x)+\frac{\rho}{2}\|x-y_k\|^2
    \right\}, \nonumber
\end{align}
where we used \(a_k+b_k=\rho(y_k-z_k)\) and the inequality comes from \cref{eq:bundle-aggregation}  and the subgradient lower bound for $h$. This model improvement is central to proving finite termination of the null steps.

\begin{remark}[Oracle complexity]
\label{rem:oracle-complexity}
The main computational steps in \Cref{alg:bundle} are updating the bundle model \(f_k\) and solving the proximal subproblem \cref{eq:pbm-subproblem-single}. Their computational cost depends on both the model \(f_k\) and the term \(h\).~A classical choice is the full cutting-plane model
$
    f_k(x)
    =
    \max_{0\leq j\leq k}
    \left\{
        f(z_j)+\innerproduct{g_j}{x-z_j}
    \right\}.
$ 
When \(h\equiv0\), the corresponding proximal subproblem is a quadratic program. For general \(h\), it becomes a structured convex optimization problem. Alternatively,~after a null step, the essential requirements \cref{eq:bundle-subgradient,eq:bundle-aggregation} can be satisfied by
the two-cut model
$    f_k(x)
    =
    \max
    \left\{
        f(z_k)+\innerproduct{g_k}{x-z_k},
        f_{k-1}(z_k)+\innerproduct{a_k}{x-z_k}
    \right\}. $ 
When \(h\equiv0\), the associated proximal subproblem admits a closed-form solution \cite[claim 1]{diaz2023optimal}.  More generally, any bundle model satisfying \cref{assump:bm-single} yields the convergence guarantees developed below, including spectral bundle models used in eigenvalue optimization and semidefinite programming \cite{helmberg2000spectral,ding2023revisiting,liao2026overview}.
We measure complexity by the total number of serious and null steps. Since each iteration requires one first-order oracle call for \(f\), these iteration bounds also give first-order oracle complexity bounds. \hfill $\square$
\end{remark}

\section{Convergence of bundle iterations under H{\"o}lder smoothness}
\label{sec:bundle-terations}

We interpret the classical PBM as a double-loop algorithm, where the null steps serve as an inner procedure for approximately solving a proximal subproblem \cref{eq:PPM-iterate}. This perspective allows us to analyze the inner bundle iterations independently and establish convergence under H{\"o}lder smoothness.

\begin{algorithm}[t]
\caption{The classical PBM as a double-loop algorithm}
\label{alg:bundle-double-loop}
\begin{algorithmic}[1]
\Require \(y_0\in\dom F\), \(\rho>0\), \(\beta\in(0,1)\),
         \(k_{\max}\in\mathbb{N}\)
\State Let \(\mathcal{S}\) be the descent test
       \cref{eq:test-PBM};
\For{\(k=0,1,\ldots,k_{\max}-1\)}
    \State \(y_{k+1}=\text{\ProDes}(y_k,\rho,\mathcal{S})\);
\EndFor
\end{algorithmic}
\end{algorithm}

\subsection{A double-loop perspective of PBMs}

At a fixed center point $y_k \in \dom F$, the null steps in \Cref{alg:bundle} can be interpreted as an inner procedure for approximating the exact proximal update \(\prox_{F/\rho}(y_k)\), with the descent test \cref{eq:PMB-test-intro} as its stopping criterion. We denote this inner procedure by \ProDes{} and write the outer update~as
\begin{equation}
\label{eq:PBM}
    y_{k+1}
    =
    \text{\ProDes}(y_k,\rho,\mathcal{S}),
\end{equation}
where \(\mathcal{S}\) denotes a stopping criterion. The resulting double-loop formulation is presented in \Cref{alg:bundle-double-loop}, with the inner procedure detailed in \Cref{alg:Proxi-descent-subproblem}.

Unlike the single-loop formulation in \Cref{alg:bundle}, which indexes both null and descent steps together, the double-loop formulation uses \(k\) for descent steps and \(j\) for inner bundle iterations. For a fixed outer index \(k\), we suppress the dependence of the inner-loop quantities on $k$. Thus, $f_j$, $F_j$, and
$z_j$ should formally be understood as $f_{k,j}$, $F_{k,j}$, and $z_{k,j}$, respectively, where
$
F_j:=f_j+h.
$ 
Initialize \(z_0=y_k\), at each inner iteration \(j\geq0\), the subroutine constructs a convex lower model \( f_j\) of \(f\) and computes 
\begin{equation}
\label{eq:PBM-trial-point}
    z_{j+1}
    =
    \argmin_{x\in\R^n}
    \left\{
        f_j(x)+h(x)
        +\frac{\rho}{2}\|x-y_k\|^2
    \right\}.
\end{equation} 
The classical descent test \cref{eq:PMB-test-intro} becomes
\begin{equation}
\label{eq:test-PBM}
    F(y_k)-F(z_{j+1})
    \geq
    \beta\bigl(F(y_k)-F_j(z_{j+1})\bigr).
\end{equation}
If \cref{eq:test-PBM} holds, the subroutine returns \(z_{j+1}\);
otherwise, it updates the bundle model according to \Cref{assump:bm-single} and continues.

The inner models \(f_j\) satisfy \cref{assump:bm-single} under the corresponding inner-loop indexing. In particular, each \(f_j\) globally minorizes \(f\) and incorporates a subgradient cut at \(z_j\). For \(j\geq1\), the optimality condition of \cref{eq:PBM-trial-point} yields
$
    a_j\in\partial f_{j-1}(z_j),
    \,
    b_j\in\partial h(z_j),
    \,
    a_j+b_j=\rho(y_k-z_j),
$ 
and the updated model retains the corresponding aggregate cut:
$
     f_j(x)
    \geq
     f_{j-1}(z_j)
    +\innerproduct{a_j}{x-z_j},
    \, \forall x\in\R^n.
$ 
Thus, when the same bundle information is retained across outer iterations, the double-loop formulation in \Cref{alg:bundle-double-loop} simply reorganizes the null and descent steps of \Cref{alg:bundle}. We next analyze the inner procedure and use its guarantees to revisit the convergence of the classical PBM in \cref{sec:revisit-classical-PBM}.

\begin{algorithm}[t]
\caption{\ProDes($y_k,\rho,\mathcal{S}$)}\label{alg:Proxi-descent-subproblem}
\begin{algorithmic}[1]
\State Initialize $z_0  = y_k$;  
\For{$j=0,1,\ldots$}
    \State Construct $f_j$ satisfying \cref{assump:bm-single};
    \State Compute $z_{j+1}$ using \cref{eq:PBM-trial-point};  
    \If{$\mathcal{S}$ holds} Break;
    \EndIf
\EndFor
 \State \textbf{Return} $z_{j+1}$;  
\end{algorithmic}
\end{algorithm}

\subsection{Basic inequalities in the subroutine \ProDes{}}
\label{subsec:complexity-inner}

We here establish basic inequalities for measuring the progress of the inner bundle iterations. Fix a proximal center \(y_k=\bar x\in\dom F\), and denote the optimal value of the true proximal subproblem~by
\begin{align}   
    \label{eq:true-prox-sub}
     F_{1/\rho}(\bar x) :=  \min_{x\in \RR^n } \; F(x) + \frac{\rho}{2} \|x - \bar x \|^2.
\end{align}
The function $F_{1/\rho}$ is also called the Moreau envelope. 
At each inner iteration $j \geq 0$, we define the following useful quantities
\begin{subequations} \label{eq:errors-j-th-iteration}
\begin{align}
m_j &:= \min_y {f}_j(y)+h(y)+\frac{\rho}{2} \|y - \bar x \| ^2, \label{eq:errors-j-th-iteration-m} \\
\delta_j&:= F_{1/\rho} (\bar x ) -  m_j, \label{eq:errors-j-th-iteration-d}\\
e_j &:= f(z_{j+1}) - {f}_j(z_{j+1}). \label{eq:errors-j-th-iteration-e}
\end{align}
\end{subequations}
Here, \(m_j\) is the optimal value of the approximate proximal subproblem. Since \( f_j\leq f\), we have \(m_j\leq F_{1/\rho}(\bar x)\), and thus \(\delta_j\) measures the
gap between the true and approximate proximal values. Finally,  $e_j$ is the model error between $f$ and its lower approximation ${f}_j$ at the candidate point $z_{j+1}$. 
Unlike the proximal value gap \(\delta_j\), which depends on the unknown \(F_{1/\rho}(\bar x)\), the model error \(e_j\) is directly computable. We will show that, if we run \ProDes{} without termination, then the sequence $m_j$ increases monotonically to $F_{1/\rho} (\bar x )$. Consequently,  $\delta_j$  converges monotonically~to~zero.   

The following lemma shows useful relationships for the quantities in \cref{eq:errors-j-th-iteration}.
\begin{lemma}
\label{prop:properties:PBM}
Consider the \(j\)-th inner iteration of \ProDes{} and the quantities defined in \cref{eq:errors-j-th-iteration}. The following statements hold. 
\begin{enumerate}[leftmargin=*]
    \item \label{properties:PBM-1}
    The model error is an upper bound of the proximal value gap:
    $
        0\leq\delta_j\leq e_j.
    $
    \item \label{properties:PBM-3}
    The suboptimality of the proximal subproblem at $z_{j+1}$ satisfies
    \begin{equation}
    \label{eq:true-proximal-suboptimality}
        F(z_{j+1}) +\frac{\rho}{2}\|z_{j+1}-\bar x\|^2 -F_{1/\rho}(\bar x) = e_j-\delta_j \leq e_j.
    \end{equation}
    Consequently, $
        \bigl\| z_{j+1}-\prox_{F/\rho}(\bar x) \bigr\|
        \leq
        \sqrt{{2e_j}/{\rho}}.$
        \item \label{properties:PBM-2}
    The candidate $z_{j+1}$ satisfies the inexact proximal inclusion
    \begin{equation}
    \label{eq:z-inexact-update}
        z_{j+1} = \bar x-\frac{1}{\rho}s_{j+1},
        \qquad
        s_{j+1}\in\partial_{e_j}F(z_{j+1}).
    \end{equation}
    More precisely, there exist $ a_{j+1}\in\partial_{e_j}f(z_{j+1}),
        \, 
        b_{j+1}\in\partial h(z_{j+1}) $
    such that $ s_{j+1} = a_{j+1}+b_{j+1} = \rho(\bar x-z_{j+1}).$
\end{enumerate}
\end{lemma}

\begin{proof}
By optimality of \(z_{j+1}\) in \cref{eq:PBM-trial-point}, we have 
$
    m_j
    =
     F_j(z_{j+1})
    +\frac{\rho}{2}\|z_{j+1}-\bar x\|^2.
$ 
Therefore,
\begin{align*}
    F(z_{j+1})
    +\frac{\rho}{2}\|z_{j+1}-\bar x\|^2
    -F_{1/\rho}(\bar x)
    &=
    F(z_{j+1})-F_j(z_{j+1})
    +m_j-F_{1/\rho}(\bar x)\\
    &=
    e_j-\delta_j.
\end{align*}
Since the left-hand side is nonnegative, it follows that
\(0\leq\delta_j\leq e_j\), proving
\cref{properties:PBM-1,eq:true-proximal-suboptimality}. The true proximal objective
$
    x\mapsto F(x)+\frac{\rho}{2}\|x-\bar x\|^2
$ 
is \(\rho\)-strongly convex. Hence, we have 
\[
    \frac{\rho}{2}
    \bigl\|
        z_{j+1}-\prox_{F/\rho}(\bar x)
    \bigr\|^2
    \leq
    F(z_{j+1})
    +\frac{\rho}{2}\|z_{j+1}-\bar x\|^2
    -F_{1/\rho}(\bar x)
    \leq e_j,
\]
which establishes $
        \bigl\|
            z_{j+1}-\prox_{F/\rho}(\bar x)
        \bigr\|
        \leq
        \sqrt{{2e_j}/{\rho}}.$

For \cref{properties:PBM-2}, the optimality condition of
\cref{eq:PBM-trial-point} and the subdifferential sum rule yield
\[
    a_{j+1}\in\partial f_j(z_{j+1}),
    \qquad
    b_{j+1}\in\partial h(z_{j+1}),
    \qquad
    a_{j+1}+b_{j+1}
    =
    \rho(\bar x-z_{j+1}).
\]
Since \(f_j\leq f\), for every \(x\in\R^n\), we have 
\begin{align*}
    f(x)
    \geq
     f_j(x)
    &\geq
     f_j(z_{j+1})
    +\innerproduct{a_{j+1}}{x-z_{j+1}}\\
    &=
    f(z_{j+1})
    +\innerproduct{a_{j+1}}{x-z_{j+1}}
    -e_j.
\end{align*}
This implies \(a_{j+1}\in\partial_{e_j}f(z_{j+1})\). Combining this inequality with
the subgradient inequality for \(h\) gives
$
    a_{j+1}+b_{j+1}
    \in
    \partial_{e_j}F(z_{j+1}),
$ 
which proves \cref{eq:z-inexact-update}.
\end{proof}

\cref{prop:properties:PBM} shows that $e_j$ is a  \textit{computable} model error for three certificates: 1) the gap between the true proximal value and the approximated proximal value, i.e., $\delta_j$; 2) the subgradient inexactness of the direction $s_{j+1}$; and 3) the suboptimality of $z_{j+1}$ to the true proximal subproblem \cref{eq:true-prox-sub}.

 As discussed above, the inner iterations of \ProDes{} progressively approximate the true proximal subproblem. We next quantify the improvement in \(m_j\) and the reduction in \(\delta_j\).

\begin{lemma}[Improvement of the inner update]
\label{lemma:improvement}
Consider the subroutine \ProDes{} in \Cref{alg:Proxi-descent-subproblem}. The quantities in
\cref{eq:errors-j-th-iteration} satisfy
\begin{subequations}
\begin{align}
    m_{j+1}-m_j
    &\geq
    \frac{\rho}{2}\|z_{j+2}-z_{j+1}\|^2,
    \label{eq:inner-improvment-m}\\
    \delta_j-\delta_{j+1}
    &\geq
    \frac{\rho}{2}\|z_{j+2}-z_{j+1}\|^2.
    \label{eq:inner-improvment-delta}
\end{align}
\end{subequations}
\end{lemma}
\begin{proof}
By the definition of $\delta_j = F_{1/\rho} (\bar x ) -  m_j$, it is clear that \cref{eq:inner-improvment-m} and \cref{eq:inner-improvment-delta} are equivalent.  It therefore suffices to prove
\cref{eq:inner-improvment-m}.

By the aggregate-cut condition, there exist
$
    a_{j+1}\in\partial f_j(z_{j+1}),
    \
    b_{j+1}\in\partial h(z_{j+1}),
    \
    a_{j+1}+b_{j+1}
    =
    \rho(\bar x-z_{j+1}),
$ 
such that
$
   f_{j+1}(x)
    \geq
   f_j(z_{j+1})
    +\innerproduct{a_{j+1}}{x-z_{j+1}}.
$ 
Combining this inequality with the subgradient inequality
$
    h(x)
    \geq
    h(z_{j+1})
    +\innerproduct{b_{j+1}}{x-z_{j+1}}
$ 
gives the composite model bound
\begin{equation}
\label{eq:assumption-3-bound}
    F_{j+1}(x)
    \geq
    F_j(z_{j+1})
    +
    \innerproduct{
        \rho(\bar x-z_{j+1})
    }{
        x-z_{j+1}
    },
    \qquad \forall x\in\R^n.
\end{equation}
Using the three-point identity
\[
    \innerproduct{
        \rho(\bar x-z_{j+1})
    }{
        x-z_{j+1}
    }
    =
    \frac{\rho}{2}\|x-z_{j+1}\|^2
    +\frac{\rho}{2}\|z_{j+1}-\bar x\|^2
    -\frac{\rho}{2}\|x-\bar x\|^2,
\]
we obtain
\begin{equation}
\label{eq:assumption-3-bound-2}
    F_{j+1}(x)
    +\frac{\rho}{2}\|x-\bar x\|^2
    \geq
    m_j
    +\frac{\rho}{2}\|x-z_{j+1}\|^2,
    \qquad \forall x\in\R^n,
\end{equation}
where we used
$
    m_j
    =
    F_j(z_{j+1})
    +\frac{\rho}{2}\|z_{j+1}-\bar x\|^2.
$ 
Setting \(x=z_{j+2}\) in
\cref{eq:assumption-3-bound-2} yields
\[
    m_{j+1}
    \geq
    m_j
    +\frac{\rho}{2}\|z_{j+2}-z_{j+1}\|^2,
\]
which proves the result.
\end{proof}

\Cref{lemma:improvement} shows that every null step monotonically increases \(m_j\) and decreases the proximal-value gap \(\delta_j\), thus providing a better approximation to the true proximal subproblem \cref{eq:true-prox-sub}. The lemma alone, however, does not show that \(\delta_j\) converges to zero. In the next subsection, we consider the class of  H{\"o}lder-smooth functions, which enables us to connect $\|z_{j+2} - z_{j+1}\|$ with $\delta_j$, and then establish its non-asymptotic convergence guarantees.

\subsection{Convergence for H{\"o}lder-smooth functions}
\label{subsec:Lipscitz} 
Recall that the function $f$ is H{\"o}lder-smooth if \cref{eq:weakly-smooth-intro} holds. For notational convenience, we define two quantities that will be used in our subsequent analysis
\begin{align} \label{eq:two-constants}
    C_{\nu} = \frac{L_\nu}{1+\nu}, \quad G_{\nu} = C_{\nu} ^{-\frac{2}{1+\nu}}.
\end{align}
It is known that \cref{eq:weakly-smooth-intro} implies that the function $f$ has an upper bound \cite{nesterov2015universal}
\begin{align}
\label{eq:weakly-smooth-conseq}
    f(y)
\le
f(x)+\innerproduct{\nabla f(x)}{y-x}
+
C_{\nu} \|y-x\|^{1+\nu}, \quad \forall x,y \in \RR^n,
\end{align}
which is a crucial inequality to establish the convergence rate. 
In particular, using \cref{eq:weakly-smooth-conseq}, we can establish a key inequality that upper bounds the error $e_j$.  
\begin{lemma}
    \label{lemma:lower-bound-step-difference}
   Consider \ProDes{} in \Cref{alg:Proxi-descent-subproblem} and the model error $e_j$ defined in \cref{eq:errors-j-th-iteration-e}. Suppose $f$ satisfies \cref{eq:weakly-smooth-intro}. Then we have $e_j\leq C_{\nu} \|z_{j+1}-z_j\|^{1+\nu}$ for all $j\ge 0$. 
\end{lemma}
\begin{proof}
    From new subgradient cut in \cref{assump:bm-single} with $y = z_{j+1}$, we have 
    $
        f_{j}(z_{j+1}) \geq f(z_j) +    \innerproduct{ \nabla f(z_j)}{z_{j+1}-z_{j}}.
    $  
    Thus, it follows that 
    \begin{align*}
         e_j =\;  f(z_{j+1}) - f_j(z_{j+1}) 
        \leq  \; &f(z_{j+1}) - (f(z_j) +    \innerproduct{ \nabla f(z_j)}{z_{j+1}-z_{j}} ) \\
        \leq \;&C_{\nu}\|z_{j+1}-z_{j}\|^{1+\nu},
    \end{align*}
    where the last inequality applies \cref{eq:weakly-smooth-conseq}.
\end{proof}
Combining \cref{prop:properties:PBM} with \cref{lemma:lower-bound-step-difference,lemma:improvement}, we immediately have the following recursion.
\begin{lemma}[Key recursion]  Consider \ProDes{} in \Cref{alg:Proxi-descent-subproblem} and the gap $\delta_j$ defined in \cref{eq:errors-j-th-iteration-d}. Suppose $f$ satisfies \cref{eq:weakly-smooth-intro}. The following recursion holds
    \label{lemma:Recursion}
    \begin{align}  \label{eq:recursion-weakly}
        \delta_{j+1} \leq \delta_j -   \frac{\rho}{2} G_{\nu} \times \delta_{j+1}^{\frac{2}{1+\nu}}, \quad \forall j \geq 0,
    \end{align}
    where $G_{\nu}$ is the constant defined in \cref{eq:two-constants}. 
\end{lemma}
\begin{proof}
    The first statement in \Cref{prop:properties:PBM} guarantees that $\delta_j \leq e_j$ for any $j \geq 0$. \Cref{lemma:lower-bound-step-difference} with the index $j+1$ gives 
    $$  \delta_{j+1}\leq e_{j+1} \leq  C_{\nu} \|z_{j+2}-z_{j+1}\|^{1+\nu}.$$
    Substituting this inequality into \cref{eq:inner-improvment-delta} from \cref{lemma:improvement}, we get the desired inequality \cref{eq:recursion-weakly}.  
\end{proof}

The key recursion \cref{eq:recursion-weakly} implies that the value gap $\delta_j$ converges monotonically to zero with a non-asymptotic rate. To quantify its rate, we use the following elementary sequence estimate.
\begin{lemma}
\label{lemma:sequence-recursion}
Let $p \in [1,2]$, $c>0$, and let $\{a_j\}_{j\geq 0}$ be a nonnegative sequence with $a_0 >0$ satisfying
\begin{equation}
\label{eq:potential-sequence-recursion}
    a_j \leq a_{j-1} -  c a_j^p,
    \qquad j\geq 1.
\end{equation}
If $p=1$, then $a_j\leq {a_0}{(1+c)^{-j}}, \, j\geq 1.$
    If $p\in(1,2]$, then we have 
    \begin{equation}
    \label{eq:potential-sequence-simple-rate}
        a_j \leq \max\left\{ a_0 e^{-j/2},\left(\frac{2}{(p-1)cj}\right)^{1/(p-1)}
        \right\}, \qquad j\geq 1.
    \end{equation}
\end{lemma}
We present a proof in \Cref{app:sequence}. \cref{lemma:sequence-recursion} is in fact the key technical ingredient in all the subsequent convergence results. The upper bound \cref{eq:potential-sequence-simple-rate} may admit different forms. In this form, the initial value $a_0$ plays a small role as its influence decays to zero exponentially fast. Note that the second term in \cref{eq:potential-sequence-simple-rate} has a uniform constant, which is independent of $a_0$ but increases as $p $ approaches $1$. From \Cref{lemma:sequence-recursion}, for any $\epsilon\in(0,a_0)$, the number of iterations required to obtain $a_j\leq\epsilon$ is at most $\mathcal{O}(
            {\log(1/\epsilon)})$ if $p=1$ and $ \mathcal{O}(\epsilon^{1-p})$ if $p \in (1,2]$. 
Applying \Cref{lemma:sequence-recursion} to
\cref{eq:recursion-weakly} with
$
    p=\frac{2}{1+\nu},
    \,
    c=\frac{\rho}{2}G_\nu,
$
gives the following convergence guarantees. 

\begin{corollary}
\label{coro:convergence-d}
Consider \ProDes{} in \Cref{alg:Proxi-descent-subproblem} and the value gap
$\delta_j$ defined in \cref{eq:errors-j-th-iteration-d}. Suppose $f$
satisfies the H{\"o}lder smoothness condition \cref{eq:weakly-smooth-intro} with $L_\nu>0$.
If $\delta_0=0$, then $\delta_j=0$ for every $j\geq 0$.
Otherwise, for every $j\geq 1$,
\begin{equation}
\label{eq:convergence-d-holder}
    \delta_j
    \leq
    \begin{cases}
        \displaystyle
        \left(1+\frac{\rho}{L_1}\right)^{-j}\delta_0,
        & \text{if } \nu=1,
        \\[1.2em]
        \displaystyle
        \max\left\{
            \delta_0e^{-j/2},
            \left(
                \frac{4(1+\nu)}
                {(1-\nu)\rho G_\nu j}
            \right)^{\frac{1+\nu}{1-\nu}}
        \right\},
        & \text{if } \nu\in[0,1).
    \end{cases}
\end{equation}
In particular, for $\nu\in[0,1)$,  we have 
$
    \delta_j
    =
   \mathcal{O}\left(
        {
            L_\nu^{\frac{2}{1-\nu}}
        }{
            \rho^{-\frac{1+\nu}{1-\nu}}
        }
        j^{-\frac{1+\nu}{1-\nu}}
    \right).
$ 
\end{corollary}

The convergence rate of \ProDes{} for the proximal-value gap $\delta_j$ automatically adapts to the smoothness of $f$. Specifically, \cref{coro:convergence-d} establishes an $\bigO(j^{-1})$ rate in the nonsmooth case $\nu=0$, which improves to $\bigO(j^{-\frac{1+\nu}{1-\nu}})$ for $\nu\in(0,1)$ and becomes linear in the smooth case $\nu=1$. The rates for $\nu=0$ and $\nu=1$ match the corresponding objective-value rates of subgradient and gradient methods for strongly convex optimization. Unlike these methods, however, \ProDes{} requires no stepsize tuning and applies uniformly across all smoothness levels. 

The proximal-value gap $\delta_j$ characterizes the progress of \ProDes{}, but it is not computable due to the unknown quantity $F_{1/\rho}(\bar x)$. In contrast, the model error
$
    e_j
    =
    f(z_{j+1})-f_j(z_{j+1})
$ 
is directly computable and satisfies $\delta_j\leq e_j$ by \cref{prop:properties:PBM}. We next establish convergence guarantees for this computable model error.

\begin{corollary}
\label{lemma:error}
Consider \ProDes{} in \Cref{alg:Proxi-descent-subproblem} and the model error $e_j$ defined in \cref{eq:errors-j-th-iteration-e}. Suppose $f$ satisfies the H{\"o}lder smoothness condition \cref{eq:weakly-smooth-intro} with $L_\nu>0$. The following statements hold.
\begin{itemize}
    \item If $\nu=1$, then
    \begin{equation}
    \label{eq:model-error-smooth}
        e_{j+1}
        \leq
        \frac{L_1}{\rho}
        \left(1+\frac{\rho}{L_1}\right)^{-j}
        \delta_0,
        \qquad j\geq 0.
    \end{equation}

    \item If $\nu\in[0,1)$, then, for every $t\geq 1$,
    \begin{align}
    \label{eq:model-error-holder}
        \min_{t+1\leq j\leq 2t}e_j
        \leq
        \left(
            \frac{2\delta_t}{\rho G_\nu t}
        \right)^{\frac{1+\nu}{2}}
        \leq
        \max\left\{
            \delta_0e^{-t/2},
            \left(
                \frac{4(1+\nu)}
                {(1-\nu)\rho G_\nu t}
            \right)^{\frac{1+\nu}{1-\nu}}
        \right\}.
    \end{align}
    In particular,
    $
        \min_{t+1\leq j\leq 2t}e_j
        =
        \mathcal{O}\left(
            {
                L_\nu^{\frac{2}{1-\nu}}
            }{
                \rho^{-\frac{1+\nu}{1-\nu}}
            }
            t^{-\frac{1+\nu}{1-\nu}}
        \right).
    $ 
\end{itemize}
\end{corollary}

\begin{proof} The proof uses the rates in \cref{coro:convergence-d}. 
We first observe the following key inequality
    \begin{align}
         \label{eq:dynamic-error}
        \frac{\rho}{2}G_{\nu} e_{j+1}^ {\frac{2}{1+\nu}}  \leq  \delta_j - \delta_{j+1},  
    \end{align}    
    which directly follows from combining \cref{eq:inner-improvment-delta} with \cref{lemma:lower-bound-step-difference}. Then, we discuss the two cases. 
    
Suppose first that $\nu=1$. Since $G_1=2/L_1$,
from \cref{eq:dynamic-error} and \cref{coro:convergence-d}, we have
$
    \frac{\rho}{L_1}e_{j+1}
    \leq
    \delta_j-\delta_{j+1}
    \leq
    \delta_j
    \leq
    \left(1+\frac{\rho}{L_1}\right)^{-j}\delta_0,
$ 
which proves \cref{eq:model-error-smooth}. 

Now suppose $\nu\in[0,1)$. Summing \cref{eq:dynamic-error} from $j=t$ to $j=2t-1$ yields
$
    \frac{\rho}{2}G_\nu
    \sum_{j=t}^{2t-1}
    e_{j+1}^{\frac{2}{1+\nu}}
    \leq
    \delta_t-\delta_{2t}
    \leq
    \delta_t.
$ 
This confirms the first inequality in \cref{eq:model-error-holder}. For the second inequality in
\cref{eq:model-error-holder}, define
\[
    C_t
    :=
    \max\left\{
        \delta_0e^{-t/2},
        \left(
            \frac{4(1+\nu)}
            {(1-\nu)\rho G_\nu t}
        \right)^{\frac{1+\nu}{1-\nu}}
    \right\}.
\]
By \cref{coro:convergence-d}, $\delta_t\leq C_t$. Moreover, we have
$
    \frac{2}{\rho G_\nu t} \leq
    \frac{4(1+\nu)}
    {(1-\nu)\rho G_\nu t}
    \leq
    C_t^{\frac{1-\nu}{1+\nu}}
$
and thus,
$
    \left(
        \frac{2\delta_t}{\rho G_\nu t}
    \right)^{\frac{1+\nu}{2}}
    \leq  \left(
        \frac{2C_t}{\rho G_\nu t}
    \right)^{\frac{1+\nu}{2}} 
    \leq 
    C_t.
$
Finally, substituting
$
    G_\nu
    =
    \left(
        \frac{L_\nu}{1+\nu}
    \right)^{-\frac{2}{1+\nu}}
$ 
gives the stated asymptotic rate.
\end{proof}

\Cref{lemma:error,prop:properties:PBM} show that \ProDes{} not only approximates the solution of \cref{eq:true-prox-sub} but also provides a computable certificate $e_j$ for the proximal subproblem accuracy. Note that \cref{eq:true-prox-sub} is strongly convex and can, in principle, be solved by standard (sub)gradient methods. However, these methods generally do not provide a comparable stopping criterion without additional information. 

The guarantee in \cref{lemma:error} applies to the last iterate when $\nu=1$, but to the best iterate over a dyadic window when $\nu\in[0,1)$. This is important because the best model error $e_j$ over the dyadic window achieves the same convergence rate as the proximal-value gap $\delta_j$. In contrast, \cref{eq:dynamic-error} only yields the following last-iterate bound:
$
    e_{j+1}
    \leq
    \left(
        \frac{2\delta_j}{\rho G_\nu}
    \right)^{\frac{1+\nu}{2}}
    =
    \bigO\left(
        j^{-\frac{(1+\nu)^2}{2(1-\nu)}}
    \right).
$ 
This is a slower rate; for example, when $\nu=0$, it becomes $\bigO(j^{-1/2})$, compared with  $\bigO(j^{-1})$ over the dyadic window.

\begin{remark} \label{remark:inner-loop-analysis}
Our analysis in this section shows the intrinsic behavior of the subroutine \ProDes{} (which corresponds to the null-step cycle of the classical PBM). In particular, we show the convergence of the subroutine \ProDes{} under the H{\"o}lder-smooth condition without any stopping criterion. In contrast, the analysis in \cite{diaz2023optimal} crucially relies on the test \cref{eq:PMB-test-intro} to establish the convergence of the null-step cycle. Our analysis is closely related to a recent work \cite{liang2026proximal} that also considers the intrinsic behavior of the bundle subroutine under the H{\"o}lder-smooth condition. However, the analysis of \cite{liang2026proximal} uses the cutting-plane model $f_{j}(\cdot) = \max_{ 1 \leq i \leq j} f(z_i) + \innerproduct{v_i}{\cdot - z_i}$ where $v_i \in \partial f(z_i)$. 
\end{remark}

\section{Revisiting classical proximal bundle methods} \label{sec:revisit-classical-PBM}

In \cref{sec:bundle-terations}, we have analyzed the convergence of \ProDes{} without specifying a stopping criterion $\mathcal{S}$. We now revisit the classical PBM in \Cref{alg:bundle} with the descent test \cref{eq:test-PBM}, and establish convergence guarantees for the composite problem \cref{eq:pb-main}. In particular, we interpret the classical PBM as the double-loop algorithm in \Cref{alg:bundle-double-loop}, with $\mathcal{S}$ given by
\cref{eq:test-PBM}, i.e.
\begin{align}
    \label{eq:PBM-classical}
    {
    y_{k+1} = \text{\ProDes{}}(y_k,\rho,\mathcal{S} = \cref{eq:test-PBM} ), \; k =0,1,\ldots.
    }
\end{align}

\subsection{Analysis outline} \label{subsection:PBM-outline}
We first explain why each call to \ProDes{} with the classical descent test \cref{eq:test-PBM} must terminate after finitely many inner iterations. For an outer iterate $y_k$, define the proximal gap
\begin{equation}
\label{eq:proximal-gap}
    \Delta_k
    :=
    F(y_k)-F_{1/\rho}(y_k),
\end{equation}
where $F_{1/\rho}$ denotes the Moreau envelope of the composite objective $F=f+h$; see \cref{eq:true-prox-sub}. This quantity measures the decrease achieved by the exact proximal update and satisfies $\Delta_k>0$ whenever $y_k\notin X^\star$.  
The following proposition shows that if the test \cref{eq:test-PBM} fails at the $j$-th iteration of \ProDes{}, then the model error $e_j$  must be larger than a fraction of $\Delta_{k}$. 

\begin{proposition}[Violation of the descent test]
\label{prop:violation-testing}
Consider \ProDes{} in \Cref{alg:Proxi-descent-subproblem} with center $y_k \notin X^\star$, and let $e_j$ be defined in \cref{eq:errors-j-th-iteration-e}. If the descent test \cref{eq:test-PBM} fails at the $j$-th iteration, then
\begin{equation}
\label{eq:violation-testing}
    e_j  >   (1-\beta) \left( F(y_k)-F_j(z_{j+1})
    \right)
    \geq  (1-\beta)\Delta_k.
\end{equation}
\end{proposition}
\begin{proof}
If \cref{eq:test-PBM} fails, then we must have
\[
   F(z_{j+1}) > F(y_k) - \beta\left(
        F(y_k)-F_j(z_{j+1})
    \right).
\]
By definition $F=f+h$ and $F_j=f_j+h$, the model error satisfies  
$
    e_j    =    f(z_{j+1})-f_j(z_{j+1})
    =
    F(z_{j+1})-F_j(z_{j+1}).
$ 
Therefore,
$
    e_j    >    (1-\beta)
    \left(
        F(y_k)-F_j(z_{j+1})
    \right).
$ 

Since $F_j\leq F$ and $z_{j+1}$ minimizes
the approximated proximal subproblem \cref{eq:PBM-trial-point}, we have 
\[
    F_j(z_{j+1})
    +
    \frac{\rho}{2}\|z_{j+1}-y_k\|^2
    \leq
    F_{1/\rho}(y_k).
\]
This implies that 
$
    F(y_k)-F_j(z_{j+1})
    \geq
    \Delta_k
    +
    \frac{\rho}{2}\|z_{j+1}-y_k\|^2
    \geq
    \Delta_k,
$ 
which completes the proof.
\end{proof}

Since the model error $e_j$ converges to zero by \cref{lemma:error} and $\Delta_k>0$ whenever $y_k\notin X^\star$, \cref{prop:violation-testing} implies that the descent test \cref{eq:test-PBM} must be satisfied within finitely many iterations. Moreover, for fixed $\rho>0$ and
$\beta\in(0,1)$, combining \cref{prop:violation-testing} with \cref{lemma:error} guarantees that the inner loop stops in at most
$
    \mathcal{O}\Big(
        \Delta_k^{-\frac{1-\nu}{1+\nu}}
    \Big)
$ iterations for $\nu\in[0,1)$. We will further show that the smooth case $\nu=1$ admits a uniformly bounded inner-loop complexity.

The overall complexity follows by combining two additional facts. First, we have
$
    \Delta_k
    \geq
    \Omega(\epsilon^2)
$
whenever $F(y_k)-F^\star\geq\epsilon$, for sufficiently small $\epsilon>0$ (see \cref{lemma:proximal-gap}). Second, we need at most $\mathcal{O}(\epsilon^{-1})$ outer iterations to obtain $F(y_k)-F^\star\leq\epsilon$ (see \cref{proposition:outer-loop-classical}). 
Consequently, with any fixed proximal parameter $\rho > 0$, the classical PBM requires at most
\begin{equation}
\label{eq:overall-complexity-PBM}
    \underbrace{
        \mathcal{O}\left(\epsilon^{-1}\right)
    }_{\text{Outer complexity}}
    \cdot
    \underbrace{
        \mathcal{O}\left(
            \epsilon^{-2\frac{1-\nu}{1+\nu}}
        \right)
    }_{\text{Inner complexity}}
    =
    \mathcal{O}\left(
        \epsilon^{-\frac{3-\nu}{1+\nu}}
    \right)
\end{equation}
total iterations to find an iterate satisfying $F(y_k)-F^\star\leq\epsilon$. This rate recovers the guarantees in \cite{diaz2023optimal} for $\nu=0$ and $\nu=1$, while
additionally covering the intermediate H{\"o}lder-smooth regime $\nu\in(0,1)$. This rate holds for a fixed $\rho$, and it can be further improved using an $\epsilon$-dependent $\rho$. 

The following subsection establishes these guarantees formally and makes their dependence on the problem and algorithm parameters explicit.

\subsection{Complexity of the classical PBM for H{\"o}lder-smooth functions}
\label{subsec:complexity-classical-PBM}

We first establish that the classical PBM in \Cref{alg:bundle-double-loop} takes at most $\mathcal{O}(\epsilon^{-1})$ outer iterations to find a point satisfying $F(y_k) - F^\star \leq \epsilon$ when $\rho$ is fixed. For this, we recall a useful fact regarding the cost value drop at each descent step, as also used in \cite[Lemma 5.1]{diaz2023optimal}. 

\begin{lemma}
    \label{lemma:outer-loop-PBM-classical}
    Consider the classical PBM in \Cref{alg:bundle-double-loop}. For each outer-loop update, i.e., when the descent test holds, we have   
    \begin{align*}
        F(y_{k+1}) \leq F(y_k) - \beta \Delta_k, \; \forall k \geq 0,
    \end{align*}
    where $\Delta_k$ is the true proximal gap defined in \cref{eq:proximal-gap}. 
\end{lemma}

\begin{proof}
   At the accepted inner iterate $y_{k+1}=z_{j+1}$, the descent test gives
   $$
   \begin{aligned}
    F(y_k)-F(y_{k+1}) \geq \beta\bigl(F(y_k)-F_j(z_{j+1})\bigr).
   \end{aligned}
   $$
Moreover, since $F_j\leq F$ and $z_{j+1}$ minimizes the model proximal subproblem,
we have $
    F_j(z_{j+1})
    +\frac{\rho}{2}\|z_{j+1}-y_k\|^2
    \leq
    F_{1/\rho}(y_k).
$ 
Thus, we know  $
    F(y_k)-F_j(z_{j+1})
    \geq
    F(y_k)-F_{1/\rho}(y_k)
    =
    \Delta_k.
$ 
Substituting this fact into the descent test above leads to the desired inequality  
$
    F(y_{k+1})\leq F(y_k)-\beta\Delta_k.
$ 
\end{proof}

This confirms that each descent step achieves a cost value drop of at least a $\beta$ portion of the proximal gap $\Delta_k$. As expected, if $\beta = 1$, this reduces to the cost improvement of the true PPM. With this cost value drop at each descent step, we have the following outer-loop complexity.  

\begin{proposition}[Outer-loop complexity of the classical PBM]
\label{proposition:outer-loop-classical}
Consider the convex optimization problem \cref{eq:pb-main} and let
$\{y_k\}$ be generated by the classical PBM in \Cref{alg:bundle-double-loop}. Suppose there is a constant $D >0$ such that $D \geq \sup \{\Dist(x,X^\star) \mid F(x) \leq F(y_0)\}$. 
Then, for any $\epsilon>0$, the method generates an iterate
$y_k$ satisfying
$
F(y_k)-F^\star\le\epsilon
$
within at most
\begin{align}
\label{eq:outer}
    T_{\mathrm{outer}}
:=
\left \lceil \frac{2\rho D^2}{\beta\epsilon} \right \rceil
+
\left \lceil \frac{\log_+\!\left((F(y_0)-F^\star)/(\rho D^2)\right)}
{\log\!\left(1/(1-\beta/2)\right)}\right \rceil
\end{align}
outer iterations, where $\log_+(\cdot) = \max \{0, \log(\cdot)\}$. 
\end{proposition}

This outer-loop guarantee requires only convexity and does not rely on any smoothness assumption. Note that in \cref{eq:outer}, choosing $\rho = \Theta (\epsilon)$ implies a complexity of $\mathcal{O}(\log(1/\epsilon))$. The proof of \cref{proposition:outer-loop-classical} is standard; see, e.g., \cite{diaz2023optimal,ruszczynski2011nonlinear}. 
 For self-completeness, we provide a proof in \cref{appendix:outer-loop-PBM}.

\begin{remark}
   Note that \cref{proposition:outer-loop-classical} requires a uniform constant $D \geq \sup \{\Dist(x,X^\star) \mid F(x) \leq F(y_0)\}$. If we use $\{y_k^\rho\}_{k\geq 0}$ to denote the sequence of descent steps with the parameter $\rho > 0$, then we can use $D_{\rho} = \sup_{k\geq 0} \Dist(y_k^\rho,X^\star)$ to replace $D$ in \cref{eq:outer}. In this case, however, we can not simply pick $\rho = \Theta (\epsilon)$ and claim that the complexity \cref{eq:outer} improves to $\bigO (\log(1/\epsilon))$, since $D_{\rho}$ depends on $\rho$ and thus $\epsilon$. The assumption $D \geq \sup \{\Dist(x,X^\star) \mid F(x) \leq F(y_0)\}$ can be satisfied for any function that has bounded sublevel sets. It can also be guaranteed for any function with H\"older growth.  
\end{remark}

We next bound the number of inner iterations required to satisfy the classical descent test by combining \cref{lemma:error} with the violation condition in \cref{prop:violation-testing}.

\begin{proposition}[Inner-loop complexity of the classical PBM]
\label{proposition:inner-loop-classical}
Consider \ProDes{} in \Cref{alg:Proxi-descent-subproblem} with center $y_k\notin X^\star$ and stopping criterion \cref{eq:test-PBM}. Suppose $f$ satisfies the H{\"o}lder smoothness \cref{eq:weakly-smooth-intro} with $L_\nu>0$. Let
$ \Delta_k $ be the proximal gap \cref{eq:proximal-gap} and $\delta_0$ be the initial value gap \cref{eq:errors-j-th-iteration-d}. 
\begin{itemize}
\setlength{\itemsep}{0pt}
    \item If $\nu=1$, the descent test \cref{eq:test-PBM} is satisfied within at most  $T_{\mathrm{inner},k}$ inner iterations with 
    \begin{equation}
    \label{eq:inner-loop-classical-smooth}
       T_{\mathrm{inner},k}
        = 2+\left\lceil
            \frac{
                \log_+\!\left(
                    \frac{L_1\delta_0}
                    {\rho(1-\beta)\Delta_k}
                \right)
            }{
                \log\!\left(
                    1+{\rho}/{L_1}
                \right)
            }
        \right\rceil. 
    \end{equation}

    \item If $\nu\in[0,1)$, define
    \begin{equation}
    \label{eq:inner-loop-classical-window}
        s_k := \left\lceil
            \max\left\{
                1,\,
                2\log_+\!\left(
                    \frac{\delta_0}
                    {(1-\beta)\Delta_k}
                \right),\,
                \frac{
                    4(1+\nu)
                }{
                    (1-\nu)\rho G_\nu
                    \bigl((1-\beta)\Delta_k\bigr)^{
                        \frac{1-\nu}{1+\nu}
                    }
                }
            \right\}
        \right\rceil.
    \end{equation}
    Then the descent test \cref{eq:test-PBM} is satisfied within at most $2s_k+1$
    inner iterations. In particular,
    \begin{equation}
    \label{eq:inner-loop-classical-holder}
        T_{\mathrm{inner},k}
        =
        \mathcal{O}\left(
            1
            +
            \log_+\!\left(
                \frac{\delta_0}
                {(1-\beta)\Delta_k}
            \right)
            +
            \frac{
                L_\nu^{\frac{2}{1+\nu}}
            }{
                \rho
                (1-\beta)^{\frac{1-\nu}{1+\nu}}
                \Delta_k^{\frac{1-\nu}{1+\nu}}
            }
        \right).
    \end{equation}
\end{itemize}
\end{proposition}

\begin{proof}
By \cref{prop:violation-testing}, failure of the descent test at the
$j$th inner iteration implies
$
    e_j  >  (1-\beta)\Delta_k.
$
Consequently, the descent test must be satisfied whenever
$
    e_j
    \leq
    (1-\beta)\Delta_k.
$ 

Suppose first that $\nu=1$. By \cref{eq:model-error-smooth} in  \cref{lemma:error},
we have $
    e_{j+1}
    \leq
    \frac{L_1}{\rho}
    \left(1+\frac{\rho}{L_1}\right)^{-j}
    \delta_0.
$ 
Solving $e_{j+1}\leq(1-\beta)\Delta_k$ for $j$ gives the bound \cref{eq:inner-loop-classical-smooth} (noting that we start $j = 0$ for the inner loop).

Now suppose $\nu\in[0,1)$. By \cref{lemma:error}, for every
$s\geq 1$, we have 
\[
    \min_{s+1\leq j\leq 2s}e_j
    \leq
    \max\left\{
        \delta_0e^{-s/2},
        \left(
            \frac{
                4(1+\nu)
            }{
                (1-\nu)\rho G_\nu s
            }
        \right)^{\frac{1+\nu}{1-\nu}}
    \right\}.
\]
The choice of $s_k$ guarantees
\begin{align*}
    \max\left\{
        \delta_0e^{-s_k/2},
        \left(
            \frac{
                4(1+\nu)
            }{
                (1-\nu)\rho G_\nu s_k
            }
        \right)^{\frac{1+\nu}{1-\nu}}
    \right\} \leq  (1-\beta)\Delta_k.
\end{align*}
Thus, the descent test must be satisfied within $2s_k+1$ inner iterations. 
Finally, using
$
    G_\nu^{-1}
    =
    \left(
        \frac{L_\nu}{1+\nu}
    \right)^{\frac{2}{1+\nu}}
$ 
in \cref{eq:inner-loop-classical-window} gives
\cref{eq:inner-loop-classical-holder}.
\end{proof}

The complexity \cref{eq:inner-loop-classical-holder} depends on the H{\"o}lder
constant and proximal parameter through
$
    {\rho^{-1}}{L_\nu^{\frac{2}{1+\nu}}}.
$ 
Also, both \cref{eq:inner-loop-classical-holder,eq:inner-loop-classical-smooth} depend on the initial value gap $\delta_0$, which may vary with different $y_k$. The following lemma bounds $\delta_0$ in terms of the proximal gap $\Delta_k$. The proof is provided in \Cref{appendix:initial-proximal-gap}.

\begin{lemma}\label{lemma:initial-proximal-gap}
Consider \ProDes{} in \Cref{alg:Proxi-descent-subproblem} with center $y_k\notin X^\star$. Suppose that $f$ satisfies \cref{eq:weakly-smooth-intro} with $\nu\in[0,1]$ and $L_\nu>0$. Let $\Delta_k$ and $\delta_0$ be defined in \cref{eq:proximal-gap} and \cref{eq:errors-j-th-iteration-d}, respectively.
Then
\begin{equation}
\label{eq:initial-gap-unified}
    \delta_0
    \leq
    \frac{1-\nu}{2(1+\nu)}\Delta_k
    +
    \frac{3+\nu}{2(1+\nu)}
    \frac{L_\nu^{\frac{2}{1+\nu}}}{\rho}
    \Delta_k^{\frac{2\nu}{1+\nu}}. 
\end{equation}
\end{lemma}

In particular, for the smooth case $\nu = 1$, \cref{eq:initial-gap-unified} reduces to
$
    \delta_0\leq (L_1/\rho)\Delta_k.
$
Consequently, \cref{eq:inner-loop-classical-smooth} implies that the number of inner iterations is uniformly bounded independently of the proximal center $y_k$. To relate the inner-loop complexity in \cref{proposition:inner-loop-classical} to the target accuracy, we next lower-bound the proximal gap $\Delta_k$ in terms of the suboptimality $\epsilon$.

\begin{lemma}[{\cite[Lemma 7.12]{ruszczynski2011nonlinear}}]
\label{lemma:proximal-gap}
Consider \cref{eq:pb-main}. Let $x^\star\in X^\star$ and
$y_k\in\dom F\setminus\{x^\star\}$. Then the proximal gap
defined in \cref{eq:proximal-gap} satisfies
\begin{equation*}
%\label{eq:proximal-gap-lower-bound}
    \Delta_k
    \geq
    \begin{cases}
        \displaystyle
        \frac{1}{2\rho}
        \left(
            \frac{
                F(y_k)-F^\star
            }{
                \|y_k-x^\star\|
            }
        \right)^2,
        &
        \text{if }
        F(y_k)-F^\star
        \leq
        \rho\|y_k-x^\star\|^2,
        \\[1.5ex]
        \displaystyle
        F(y_k)-F^\star
        -
        \frac{\rho}{2}\|y_k-x^\star\|^2,
        &
        \text{otherwise}.
    \end{cases}
\end{equation*}
\end{lemma}

This result follows directly from the definition of the Moreau envelope and convexity of $F$. To clarify its implication for the target accuracy, suppose that $\Dist(y_k,X^\star)\leq D$ and choose $x^\star$ as a projection of $y_k$ onto $X^\star$. Writing
$
    \Phi_k:=F(y_k)-F^\star,
$
\cref{lemma:proximal-gap} implies
\begin{equation}
\label{eq:proximal-gap-objective-gap}
    \Delta_k
    \geq
    \min\left\{
        \frac{\Phi_k^2}{2\rho D^2},
        \frac{\Phi_k}{2}
    \right\}.
\end{equation}
Consequently, whenever $\Phi_k\geq\epsilon$, we must have 
\begin{equation}
\label{eq:proximal-gap-objective-gap-2}
    \Delta_k
    \geq \min\left\{
        \frac{\epsilon^2}{2\rho D^2},
        \frac{\epsilon}{2}
    \right\} = 
    \frac{
        \epsilon^2
    }{
        2\max\{\rho D^2,\epsilon\}
    }.
\end{equation}
In particular, if $\epsilon\leq\rho D^2$, then $
    \Delta_k
    \geq
    {\epsilon^2}/({2\rho D^2}).
$

We are now ready to establish the overall complexity of the classical
PBM by combining the inner- and outer-loop estimates in
\Cref{proposition:outer-loop-classical,proposition:inner-loop-classical} with \cref{lemma:initial-proximal-gap}.

\begin{theorem}[Overall complexity of the classical PBM]
\label{theorem:complexity-classical-PBM}
Assume the setting of \Cref{proposition:outer-loop-classical}, and suppose that $f$
satisfies \cref{eq:weakly-smooth-intro} with $\nu\in[0,1]$ and $L_\nu>0$. For any $\epsilon\in(0,F(y_0)-F^\star),$ let $T_{\mathrm{total}}$ denote the total number of inner iterations required by the classical PBM in \Cref{alg:bundle-double-loop} to obtain an iterate satisfying
$
    F(y_k)-F^\star\leq\epsilon.
$
The following statements hold.

\begin{itemize}
    \item If $\nu=1$, then
    \begin{equation}
    \label{eq:total-complexity-classical-smooth}
    \begin{aligned}
        T_{\mathrm{total}}
        &\leq
        T_{\mathrm{outer}} \times
        \left[
            2+
            \left\lceil
                \frac{
                    \log_+\left(
                        \frac{L_1^2}{(1-\beta)\rho^2}
                    \right)
                }{
                    \log\left(
                        1+\frac{\rho}{L_1}
                    \right)
                }
            \right\rceil
        \right].
    \end{aligned}
    \end{equation}

    \item If $\nu\in[0,1)$, then
    \begin{equation}
    \label{eq:total-complexity-classical-holder}
    \begin{aligned}
        T_{\mathrm{total}}
        &\leq
        T_{\mathrm{outer}} \times 
        \Bigg[
            3 + 4\log\left( \frac{1}{1-\beta} \right)
            +\frac{8}{1-\nu}\frac{L_\nu^{\frac{2}{1+\nu}}
            }{\rho(1-\beta)^{\frac{1-\nu}{1+\nu}}}
            \left(
                \frac{
                    2\max\{\rho D^2,\epsilon\}
                }{
                    \epsilon^2
                }
            \right)^{\frac{1-\nu}{1+\nu}}
        \Bigg].
    \end{aligned}
    \end{equation}
\end{itemize}
Here $T_{\mathrm{outer}}$ is defined in \cref{eq:outer}.
\end{theorem}

\begin{proof}
By \Cref{proposition:outer-loop-classical}, the method requires at most $T_{\mathrm{outer}}$ outer iterations. It therefore suffices to bound the number of inner iterations at each proximal center $y_k$ satisfying $F(y_k)-F^\star>\epsilon$.

\textbf{Case 1:} $\nu=1$. 
\Cref{lemma:initial-proximal-gap} implies  
$
    \delta_0 \leq \frac{L_1}{\rho}\Delta_k.
$ 
Substituting this bound into
\cref{eq:inner-loop-classical-smooth} gives
\begin{equation} \label{eq:inner-complexity-per-iteration}
    T_{\mathrm{inner},k}
    \leq
    2+ \left\lceil
        \frac{
            \log_+\left(
                \frac{L_1^2}{(1-\beta)\rho^2}
            \right)
        }{
            \log\left(
                1+\frac{\rho}{L_1}
            \right)
        }
    \right\rceil.
\end{equation}
Multiplying by $T_{\mathrm{outer}}$ establishes
\cref{eq:total-complexity-classical-smooth}.

\textbf{Case 2:} $\nu\in[0,1)$. For notational simplicity, define 
\begin{equation} \label{eq:constant-for-proof}
r:=\frac{1-\nu}{1+\nu} \in (0,1], \quad
A_k := \frac{L_\nu^{\frac{2}{1+\nu}}}{\rho\Delta_k^{r}}, 
   \quad B_k:= \frac{ 4(1+\nu)}{(1-\nu)\rho G_\nu
        ((1-\beta)\Delta_k)^r}
\end{equation}
where $B_k$ denotes the last term in \cref{eq:inner-loop-classical-window}. Some elementary calculations confirm that $B_k\geq4A_k$. Meanwhile, \Cref{lemma:initial-proximal-gap} guarantees that 
\[
    \frac{\delta_0}{\Delta_k}
    \leq \frac{r}{2} + \left(1+\frac{r}{2}\right)A_k
    \leq 1+2A_k,
\]
where the last inequality uses $r\in(0,1]$. Consequently, we have 
\begin{equation}
\label{eq:initial-gap-log-bound}
    \log_+\left(
        \frac{\delta_0}{(1 - \beta)\Delta_k}
    \right)
    \leq
    \log\left(\frac{1}{1 - \beta}\right)
    +
    \log(1+2A_k)
    \leq
    \log\left(\frac{1}{1 - \beta}\right)
    +
    2A_k,
\end{equation}
where the last inequality comes from the fact that $t \mapsto \log(t)$ is concave for $t > 0$. By \cref{eq:inner-loop-classical-window} in \Cref{proposition:inner-loop-classical}, the descent test is satisfied
within at most $2s_k+1$ inner iterations, where
\[
    s_k
    :=
    \left\lceil
        \max\left\{
            1,\,
            2\log_+\left(
                \frac{\delta_0}{(1 - \beta)\Delta_k}
            \right),
            B_k
        \right\}
    \right\rceil.
\]
From \cref{eq:initial-gap-log-bound} and $B_k \geq 4A_k$, we have 
$
    s_k  \leq  1 +2\log\left(\frac{1}{1 - \beta}\right)+B_k.
$ 
Hence, substituting $B_k$ and the definition $G_\nu$ in \cref{eq:two-constants}, we get 
\begin{equation}
    \begin{aligned}
    \label{eq:T-inner}
    T_{\mathrm{inner},k}
    &\leq
    2s_k+1 
    \leq
    3 + 4\log\left(\frac{1}{1 - \beta}\right) + \frac{8}{1-\nu}
    \frac{
        L_\nu^{\frac{2}{1+\nu}}
    }{\rho(1 - \beta)^r\Delta_k^r}. 
\end{aligned}
\end{equation}
Finally, before reaching an
$\epsilon$-optimal solution, we have 
$
    \Delta_k^{-1}
    \leq  2\max\{\rho D^2,\epsilon\}  \epsilon^{-2}
$  by \cref{eq:proximal-gap-objective-gap-2}. 
Substituting this bound and recalling that $r=(1-\nu)/(1+\nu)$ proves \cref{eq:total-complexity-classical-holder}.
\end{proof}

\Cref{theorem:complexity-classical-PBM} applies to any fixed proximal parameter $\rho > 0$. To the best of our knowledge, it provides the first complexity guarantees of the classical PBM with the descent test \cref{eq:test-PBM} for the full class of H{\"o}lder-smooth functions. %Our results match the state-of-the-art analysis 
When $\rho > 0$ is fixed and  independent of the target accuracy $\epsilon$, the overall complexity \cref{eq:total-complexity-classical-smooth} is $\mathcal{O}(\epsilon^{-1})$ for the smooth case with $\nu = 1$, which matches the usual rate of gradient descent; the overall complexity \cref{eq:total-complexity-classical-holder} becomes $\mathcal{O}(\epsilon^{-3})$ for the nonsmooth case with $\nu = 0$, which is the same as the state-of-the-art analysis \cite{kiwiel2000efficiency,diaz2023optimal}. For the intermediate regime $\nu\in(0,1)$, the overall complexity is $\mathcal{O}\left(
        \epsilon^{-\frac{3-\nu}{1+\nu}}    \right)$  as highlighted in \cref{eq:overall-complexity-PBM}, which has not been established previously. Note that we can obtain a stronger guarantee by selecting the proximal parameter according to the target accuracy. Specifically, if $\rho=\Theta(\epsilon)$, then the inner complexity improves to
$
    \mathcal{O}\left(
        \rho^{-1}\epsilon^{-\frac{1-\nu}{1+\nu}}
    \right)
    =
    \mathcal{O}\left(
        \epsilon^{-\frac{2}{1+\nu}}
    \right),
    \, \nu\in[0,1),
$ 
while the outer complexity \cref{eq:outer} becomes $\mathcal{O}(\log(1/\epsilon))$. Consequently, the overall complexity is
$
    \widetilde{\mathcal{O}}\left(
        \epsilon^{-\frac{2}{1+\nu}}
    \right),
    \, \nu\in[0,1),
$ 
matching the rate of the universal primal gradient method \cite{nesterov2015universal} up to a logarithmic factor. 

A refined analysis of the inner-iteration counting removes this logarithmic factor.

\begin{theorem} 
\label{thm:classical-improved-explicit}
Assume the setting of \cref{theorem:complexity-classical-PBM} and let $\nu\in[0,1)$ and $ \epsilon\in(0,F(y_0)-F^\star).$ Pick  $\rho=c\epsilon/D^2$ for any fixed $ c>0$.
Then the classical PBM in \Cref{alg:bundle-double-loop} finds an iterate $y_k$ satisfying $ F(y_k)-F^\star\leq \epsilon $ 
within the following total iterations
$\mathcal O\!\left(
\frac{
    L_\nu^{\frac{2}{1+\nu}} D^2
}{
    \beta (1-\beta)^{\frac{1-\nu}{1+\nu}}
    \epsilon^{\frac{2}{1+\nu}}
}
\right)$.
\end{theorem}

\begin{proof}
Define
$
    \Phi_k:=F(y_k)-F^\star,
     r:=\frac{1-\nu}{1+\nu}\in(0,1], a := \max\{ 1,c\} $, and $
    \gamma:=1-\frac{\beta}{2a}\in(0,1).
$
From \cref{eq:T-inner}, the number of inner iterations at the
$k$th proximal center satisfies
\[
    T_{\mathrm{inner},k}
    \leq
    C_\beta
    +
    \frac{8}{1-\nu}
    \frac{L_\nu^{\frac{2}{1+\nu}}}
    {\rho(1-\beta)^r\Delta_k^r},
    \qquad
C_\beta:=3+4\log\left(\frac{1}{1-\beta}\right).
\]
To get the total number of iterations, we sum up all the inner steps, i.e., 
\begin{align}
\label{eq:total}
        T_{\mathrm{total}}
    =
    \sum_{k=0}^{N-1}T_{\mathrm{inner},k} \leq  C_\beta N +  \frac{8}{1-\nu}
    \frac{L_\nu^{\frac{2}{1+\nu}}}
    {\rho(1-\beta)^r} \sum_{k=0}^{N-1}\Delta_k^{-r}
\end{align}
where 
$
    N:=\min\{k\geq 0:\Phi_k\leq\epsilon\}.
$
It therefore suffices to bound
$\sum_{k=0}^{N-1}\Delta_k^{-r}$.
By \cref{eq:proximal-gap-objective-gap} and the choice of $\rho$, for all $k = 0, \ldots, N-1$, it holds that
\begin{equation*}
    \Delta_k
    \geq
    \min\left\{
        \frac{\Phi_k}{2},
        \frac{\Phi_k^2}{2\rho D^2}
    \right\} = \min\left\{
        \frac{\Phi_k}{2},
        \frac{\Phi_k^2}{2 c \epsilon}
    \right\} \geq   \frac{\Phi_k}{2a}.
\end{equation*}
Combining this inequality with \cref{lemma:outer-loop-PBM-classical} yields
$
    \Phi_{k+1}\leq\gamma\Phi_k.
$
This leads to
$
     \epsilon
    <
    \Phi_{N-1}
    \leq
    \gamma^{N-1-k}\Phi_k,  k=0,\ldots,N-1,
$ 
and consequently
$
    \Delta_k^{-r}
    \leq
    \left({2a}/{\epsilon}\right)^r
    \gamma^{r(N-1-k)}, k=0,\ldots,N-1.
$ 
Thus, it follows that
\begin{equation}
\label{eq:sum-large-gap-rho}
    \sum_{k=0}^{N-1}\Delta_k^{-r}
    \leq
    \frac{(2a/\epsilon)^r}{1-\gamma^r} \leq \left ( \frac{2a}{\epsilon}\right)^r  \frac{2a}{r\beta}
\end{equation}
where the last inequality follows from the inequality $r(1-t) \le 1 -t^r $ with $t = \gamma =  1- \frac{\beta}{2a}$, which comes 
from the fact that the function \(t  \mapsto t^r\) is concave on \((0,\infty)\).

Finally, \cref{eq:outer} gives
$
    N
    \leq
    \left\lceil
        {2 c }/{\beta}
    \right\rceil
    +
    \left\lceil
        \frac{
            \log_+\!\left(
                \frac{F(y_0)-F^\star}{ c \epsilon}
            \right)
        }{
            \log\!\left(1/(1-\beta/2)\right)
        }
    \right\rceil
$
and plugging \cref{eq:sum-large-gap-rho} in \cref{eq:total} gives 
\begin{align*}
        T_{\mathrm{total}}
     \leq & \left ( 3+4\log\left(\frac{1}{1-\beta}\right)\right) \left(\left\lceil
        \frac{2 c }{\beta}
    \right\rceil
    +
    \left\lceil
        \frac{
            \log_+\!\left(
                \frac{F(y_0)-F^\star}{ c \epsilon}
            \right)
        }{
            \log\!\left(1/(1-\beta/2)\right)
        }
    \right\rceil \right) +  \frac{8}{1-\nu}
    \frac{L_\nu^{\frac{2}{1+\nu}} D^2}
    { c \epsilon (1-\beta)^r} \left ( \frac{2a}{\epsilon}\right)^r  \frac{2a}{r\beta}.
\end{align*}
Thus, as $\epsilon \downarrow 0$, the complexity is simplified as 
$
     T_{\mathrm{total}} = \mathcal{O}\!\left(
\frac{
    L_\nu^{\frac{2}{1+\nu} }D^2
}{
    \beta (1-\beta)^{\frac{1-\nu}{1+\nu}}
    \epsilon^{\frac{2}{1+\nu}}
}
\right).
$
\end{proof}

\subsection{Comparison with the state-of-the-art rates}
\label{subsec:comparison-state-of-art}

\begin{table}[t]
\centering
\caption{Comparison of the total iteration complexity of the classical PBM for H{\"o}lder-smooth functions satisfying \cref{eq:weakly-smooth-intro}. The fixed-$\rho$ columns report the dominant terms as $\epsilon\downarrow 0$. Here, $D\geq \sup\{\Dist(x,X^\star)\mid F(x)\leq F(y_0)\}$, and $r_\nu:=(1-\nu)/(1+\nu)$. ``N/A'' indicates that the corresponding rate is unavailable.}
\label{tab:comparison-classical-PBM}
\renewcommand{\arraystretch}{2.1}
\resizebox{\textwidth}{!}{%
\begin{tabular}{@{}lcccc@{}}
\toprule
&
\multicolumn{2}{c}{This work}
&
\multicolumn{2}{c}{D\'iaz and Grimmer \cite{diaz2023optimal}}
\\
\cmidrule(lr){2-3}
\cmidrule(lr){4-5}
Function class
&
Fixed $\rho>0$
&
$\rho=\epsilon/D^2$
&
Fixed $\rho>0$
&
$\rho=\epsilon/D^2$
\\
\midrule

Smooth
$\left(\nu=1\right)$
&
$\displaystyle
\mathcal{O}\left(
    \frac{\rho D^2}{\beta\epsilon}
    \left[
        1+
        \frac{
            \log_+\!\left(
                \frac{L_1^2}{(1-\beta)\rho^2}
            \right)
        }{
            \log(1+\rho/L_1)
        }
    \right]
\right)$
&
$\displaystyle
\widetilde{\mathcal{O}}\left(
    \frac{
        L_1D^2
    }{
        \beta\epsilon
    }
    \left[
        1+
        \log_+\!\left(
            \frac{L_1^2D^4}
                 {(1-\beta)\epsilon^2}
        \right)
    \right]
\right)$
&
$\displaystyle
\mathcal{O}\left(
    \frac{
        (L_1+\rho)^3D^2
    }{
        \beta(1-\beta)^2\rho^2\epsilon
    }
\right)$
&
$\displaystyle
\tilde{\mathcal{O}}\left(
    \frac{
        L_1^3D^6
    }{
        \beta(1-\beta)^2\epsilon^3
    }
\right)$
\\

\midrule

H{\"o}lder smooth
$\left(0<\nu<1\right)$
&
$\displaystyle
\mathcal{O}\left(
    \frac{
        L_\nu^{\frac{2}{1+\nu}}
        \rho^{r_\nu}
        D^{\frac{4}{1+\nu}}
    }{
        \beta
        (1-\beta)^{r_\nu}
        \epsilon^{\frac{3-\nu}{1+\nu}}
    }
\right)$
&
$\displaystyle
\mathcal{O}\left(
    \frac{
        L_\nu^{\frac{2}{1+\nu}}D^2
    }{
        \beta
        (1-\beta)^{r_\nu}
        \epsilon^{\frac{2}{1+\nu}}
    }
\right)$
&
N/A
&
N/A
\\

\midrule

Nonsmooth
$\left(\nu=0\right)$
&
$\displaystyle
\mathcal{O}\left(
    \frac{
        \rho L_0^2D^4
    }{
        \beta(1-\beta)\epsilon^3
    }
\right)$
&
$\displaystyle
\mathcal{O}\left(
    \frac{
        L_0^2D^2
    }{
        \beta(1-\beta)\epsilon^2
    }
\right)$
&
$\displaystyle
\mathcal{O}\left(
    \frac{
        \rho L_0^2D^4
    }{
        \beta(1-\beta)^2\epsilon^3
    }
\right)$
&
$\displaystyle
\mathcal{O}\left(
    \frac{
        L_0^2D^2
    }{
        \beta(1-\beta)^2\epsilon^2
    }
\right)$
\\

\bottomrule
\end{tabular}%
}
\end{table}

Optimal rates for classical PBMs with the descent test \cref{eq:test-PBM} have been obtained in \cite{diaz2023optimal}. The analysis  \cite{diaz2023optimal} only considers unconstrained optimization, corresponding to $h\equiv 0$. They considered generalizing the analysis beyond the unconstrained setting as future work. 

In contrast, our results apply to the composite objective $F=f+h$, including the constrained setting when $h$ is an indicator function of a convex set. 
Beyond this broader setting, our analysis covers the full class of H{\"o}lder smooth~functions. 
Even in the standard smooth and nonsmooth cases (i.e., $\nu \in \{0,1\}$), \Cref{theorem:complexity-classical-PBM} improves the dependence on the algorithmic parameters ($L_\nu, \rho, \beta$) compared with \cite{diaz2023optimal}, as shown in \Cref{tab:comparison-classical-PBM}. These improvements come from our new analysis for the bundle iterations in \Cref{sec:bundle-terations}, particularly \cref{lemma:error}; also see \Cref{remark:inner-loop-analysis}.  

For the smooth case: $\nu=1$, for any $\rho>0$, \Cref{theorem:complexity-classical-PBM} bounds the number of inner iterations per proximal center by \cref{eq:inner-complexity-per-iteration}. 
In contrast, \cite[Theorem~2.2]{diaz2023optimal} bounds the number of null steps per descent step by $\frac{16(L_1+\rho)^3}{(1-\beta)^2\rho^3}$. This reflects the rates in the first row of \Cref{tab:comparison-classical-PBM}. Our rate is more robust to $\beta$ and the proximal parameter $\rho$. For example, if $\rho\ll L_1$ (corresponding to a large step size $1/\rho$), the leading overall complexities become
\[
    \mathcal{O}\left(
        \frac{L_1D^2}{\beta\epsilon}
        \log\left(
            \frac{L_1^2}{(1-\beta)\rho^2}
        \right)
    \right)
    \qquad\text{and}\qquad
    \mathcal{O}\left(
        \frac{L_1^3D^2}
        {\beta(1-\beta)^2\rho^2\epsilon}
    \right),
\]
for our result and that of \cite{diaz2023optimal}, respectively. The polynomial term $(L_1/\rho)^2$ is replaced by $ \log(\frac{L_1^2}{\rho^2})$, and the quadratic dependence on $(1-\beta)^{-1}$ is improved to a logarithmic dependence. 
For the nonsmooth case $\nu=0$, for fixed $\rho>0$, the leading overall complexity in \Cref{theorem:complexity-classical-PBM} becomes
$
    \mathcal{O}\left(
        \frac{\rho L_0^2D^4}
        {\beta(1-\beta)\epsilon^3}
    \right),
$
whereas \cite[Theorem~2.1]{diaz2023optimal} gives
$
    \mathcal{O}\left(
        \frac{\rho L_0^2D^4}
        {\beta(1-\beta)^2\epsilon^3}
    \right).
$ 
Thus, both analyses establish the classical $\mathcal{O}(\epsilon^{-3})$ rate, but our bound improves the dependence on $(1-\beta)^{-1}$ from quadratic to linear. With $\rho=\epsilon/D^2$, both approaches attain the optimal $\mathcal{O}(\epsilon^{-2})$ rate. 

\begin{remark}[Comparison with universal gradient methods]
When $\rho=\Theta( \frac{\epsilon}{D^2})$, our refined analysis shows that the classical PBM attains
$
    \mathcal{O}\Big({\epsilon^{-\frac{2}{1+\nu}}}
    \Big), \nu\in[0,1)$. This matches the rate of the universal primal gradient method \cite{nesterov2015universal} for each fixed $\nu\in[0,1)$. Both methods adapt to the unknown H{\"o}lder exponent $\nu$ and constant $L_\nu$; the universal primal gradient method uses a line-search procedure, whereas the classical PBM uses a proximal parameter that remains fixed throughout the iterations.
Note that the universal fast gradient method \cite{nesterov2015universal} achieves the accelerated rate
$
    \mathcal{O}\left(
        \epsilon^{-\frac{2}{1+3\nu}}
    \right).
$ 
A corresponding universally accelerated guarantee remains open for the classical PBM. Achieving such a rate likely requires an acceleration mechanism, such as extrapolation or an accelerated inexact proximal-point scheme, beyond the standard bundle update; see \cite{liao2025accelerated,fersztand2025acceleration} for recent discussions. \hfill $\square$
\end{remark}

\section{Proximal bundle methods with an absolute error test} \label{sec:PBM-variant}

The classical PBMs use the descent test \cref{eq:PMB-test-intro}, equivalently written as \cref{eq:test-PBM}, to determine when the current proximal center should be moved. This test ensures that every center update yields a drop in the objective value. Such monotone descent, however, is not the only criterion for terminating the inner bundle loop and updating the proximal center. In this section, we consider an absolute error test that instead directly controls the accuracy in solving the proximal subproblem \cref{eq:true-prox-sub}. 

\subsection{A PBM variant with an absolute error test}

As discussed in \Cref{sec:bundle-terations}, the inner bundle iterations of \ProDes{} asymptotically solve the true proximal subproblem \cref{eq:true-prox-sub}. At each iteration, \ProDes{} generates a computable model error $e_j$ that quantifies the approximation quality; see \Cref{prop:properties:PBM}. Moreover, this error converges to zero by \Cref{lemma:error}. This perspective plays a key role in our analysis of the classical PBM in \Cref{subsection:PBM-outline} and leads to the refined complexity guarantees in \Cref{theorem:complexity-classical-PBM}.

From \Cref{eq:true-proximal-suboptimality} in \Cref{prop:properties:PBM}, the model error $e_j$ upper bounds the suboptimality of solving \cref{eq:true-prox-sub}, which directly relates to the inexactness \cref{eq:inexact-func}. More importantly, the inexact proximal inclusion \cref{eq:z-inexact-update} ensures an implicit inexact subgradient update $z_{j+1}
        =        \bar x-\frac{1}{\rho}s_{j+1}$  where $s_{j+1}$ is an $e_j$-inexact subgradient of $F$ at the next point $z_{j+1}$. 
These properties motivate terminating the inner bundle loop once the model error $e_j$ is sufficiently small. In particular, given an outer proximal center $y_k$, we terminate  \ProDes{} in \Cref{alg:Proxi-descent-subproblem} when 
\begin{equation} \label{eq:stop-rule-new} 
    e_j \leq \epsilon_k,
\end{equation}
where $\{\epsilon_k\}_{k\geq 0}$ is a prescribed sequence of positive inexactness. Unlike the relative descent test \cref{eq:PMB-test-intro}, this stopping criterion \cref{eq:stop-rule-new}  is based on an absolute error $\epsilon_k$. The resulting PBM variant is   
\begin{align}
    \label{eq:PBM-new}
    {
    y_{k+1} = \text{\ProDes{}}(y_k,\rho,\mathcal{S} = \cref{eq:stop-rule-new} ), \; k =0,1,\ldots,
    }
\end{align}
where $\rho > 0$ is the proximal parameter. Under the stopping rule \cref{eq:stop-rule-new}, each outer update satisfies
\begin{align}
    \label{eq:pbm-ippm-constant}
    y_{k+1}=y_k-\frac1\rho v_k,
\qquad
v_k\in \partial_{\epsilon_k}F(y_{k+1}).
\end{align}

The update \cref{eq:pbm-ippm-constant} can be interpreted as an inexact proximal point iteration and reduces to the exact PPM \cref{eq:PPM-iterate} when $\epsilon_k=0$. This type of inexact PPM scheme has also been studied in \cite{salzo2012inexact,barre2020principled,cominetti1997coupling}, but they typically do not specify an inner procedure for computing the required approximate proximal updates. The bundle iterations of \ProDes{} provide a constructive realization of \cref{eq:pbm-ippm-constant}. 
We list this PBM variant as a double-loop scheme in \Cref{alg:bundle-double-loop-varaint}. As listed in \Cref{alg:bundle-double-loop-varaint,alg:bundle-double-loop}, the key difference compared to the classical PBM is the inner-loop termination criterion. 

\begin{algorithm}[t]
\caption{A PBM variant with an absolute error test}
\label{alg:bundle-double-loop-varaint}
\begin{algorithmic}[1]
\Require \(y_0\in\dom F\), \(\rho>0\), \(\{\epsilon_k\}_{k\geq0}\),
         \(k_{\max}\in\mathbb{N}\)
\State Let \(\mathcal{S}\) be the absolute error test
       \cref{eq:stop-rule-new};
\For{\(k=0,1,\ldots,k_{\max}-1\)}
    \State \(y_{k+1}=\text{\ProDes}(y_k,\rho,\mathcal{S})\);
\EndFor
\end{algorithmic}
\end{algorithm}

\subsection{Complexity of the PBM variant}
Based on the inner-loop analysis in \Cref{sec:bundle-terations}, it is relatively easy to derive the iteration complexity of the PBM variant in \Cref{alg:bundle-double-loop-varaint}. In particular, to achieve an $\epsilon$-suboptimality, it suffices to use $\epsilon_k \leq \epsilon/2$ in \cref{eq:stop-rule-new}, as highlighted in \cref{eq:new-test}. We have the following guarantee for the outer-loop complexity. The proof is a simple extension of the standard analysis of the exact PPM \cite[Theorem 7.13]{ruszczynski2011nonlinear}, and we provide it to highlight its elegance and for completeness. 

\begin{proposition}[Outer-loop complexity of \Cref{alg:bundle-double-loop-varaint}]\label{prop:iPPM-constant}
Consider the inexact PPM \cref{eq:pbm-ippm-constant}. Fix $\epsilon>0$, and suppose that $\epsilon_k\leq\epsilon/2$ for all $k\geq0$. Then, for any $x^\star\in X^\star$ and every $N\geq1$, we have 
\begin{equation}
\label{eq:outer-complexity-absolute}
    \min_{0\leq i\leq N-1}
    \left\{F(y_{i+1})-F^\star\right\}
    \leq
    \frac{\rho\|y_0-x^\star\|^2}{2N}
    +
    \frac{\epsilon}{2}.
\end{equation}
\end{proposition}

\begin{proof}
Since $v_k\in\partial_{\epsilon_k}F(y_{k+1})$, the definition of the
$\epsilon_k$-subdifferential gives
\[
    F(y_{k+1})-F^\star
    \leq
    \langle v_k,y_{k+1}-x^\star\rangle+\epsilon_k.
\]
Substituting $v_k=\rho(y_k-y_{k+1})$ from \cref{eq:pbm-ippm-constant} and applying the  identity $\|y_k-x^\star\|^2 = 
        \|y_{k+1}-x^\star\|^2
        + 2\langle y_k-y_{k+1},y_{k+1}-x^\star\rangle + \|y_k-y_{k+1}\|^2 $, we obtain
\begin{align*}
    F(y_{k+1})-F^\star
    &\leq
    \rho\langle y_k-y_{k+1},y_{k+1}-x^\star\rangle
    +
    \epsilon_k \\
    &=
    \frac{\rho}{2}
    \left(
        \|y_k-x^\star\|^2
        -
        \|y_{k+1}-x^\star\|^2
        -
        \|y_{k+1}-y_k\|^2
    \right)
    +
    \epsilon_k.
\end{align*}
Telescoping over $k=0,\ldots,N-1$ and discarding the nonpositive terms
yields
\begin{equation} \label{eq:average-inexatness}
    \sum_{k=0}^{N-1}
    \left(F(y_{k+1})-F^\star\right)
    \leq
    \frac{\rho}{2}\|y_0-x^\star\|^2
    +
    \sum_{k=0}^{N-1}\epsilon_k.
\end{equation}
The result follows by dividing both sides by $N$ and using
$\epsilon_k\leq\epsilon/2$.
\end{proof}

As shown in \Cref{eq:average-inexatness}, despite the inexactness $\epsilon_k$ at each iteration, the accumulated error is averaged across the outer iterations. Consequently, choosing $\epsilon_k=\epsilon/2$ ensures that, for any fixed proximal parameter $\rho>0$, the PBM variant in \Cref{alg:bundle-double-loop-varaint} finds an $\epsilon$-optimal solution within $\mathcal{O}(1/\epsilon)$ outer iterations. This matches the outer-loop complexity of the classical PBM in \Cref{proposition:outer-loop-classical}, and its proof appears simpler. It is worth highlighting two key differences. First, \cref{prop:iPPM-constant} requires only convexity of $F$ and does not rely on the bounded-level-set constant $D$ used in \Cref{proposition:outer-loop-classical}.   Second, it guarantees the performance of the best iterate, or a suitable average iterate, but does not ensure monotonicity of the objective values. In contrast, the classical PBM is a descent method, i.e., each outer iteration reduces the cost value, and \Cref{proposition:outer-loop-classical} guarantees the last-iterate~performance.

Comparing the stopping criterion \cref{eq:stop-rule-new} with the violation condition \cref{eq:violation-testing}, we see that the inner-loop complexity of \cref{eq:PBM-new} follows directly from \Cref{proposition:inner-loop-classical} by replacing $(1-\beta)\Delta_k$ with $\epsilon_k$. For completeness, we state the corresponding bounds below.

\begin{proposition}[Inner-loop complexity under \cref{eq:stop-rule-new}]
\label{proposition:inner-loop-new}
Consider \ProDes{} in \Cref{alg:Proxi-descent-subproblem} with center $y_k\notin X^\star$ and stopping criterion \cref{eq:stop-rule-new}. Suppose $f$ satisfies the H{\"o}lder smoothness \cref{eq:weakly-smooth-intro} with $L_\nu>0$. Let $\delta_{0}$ denote the initial proximal-value gap associated with the center $y_k$, defined in \cref{eq:errors-j-th-iteration-d}. 
\begin{itemize}
\setlength{\itemsep}{0pt}
    \item If $\nu=1$, the test \cref{eq:stop-rule-new} is satisfied within at most  $T_{\mathrm{inner},k}$ inner iterations with 
    \begin{equation}
    \label{eq:inner-loop-smooth-new}
       T_{\mathrm{inner},k}
        = 2+\left\lceil
            \frac{
                \log_+\!\left(
                    \frac{L_1\delta_0}
                    {\rho\epsilon_k}
                \right)
            }{
                \log\!\left(
                    1+{\rho}/{L_1}
                \right)
            }
        \right\rceil.
    \end{equation}

    \item If $\nu\in[0,1)$, define
    \begin{equation}
    \label{eq:inner-loop-window-new}
        s_k := \left\lceil
            \max\left\{
                1,\,
                2\log_+\!\left(
                    \frac{\delta_0}
                    {\epsilon_k}
                \right),\,
                \frac{
                    4(1+\nu)
                }{
                    (1-\nu)\rho G_\nu
                    \bigl(\epsilon_k\bigr)^{
                        \frac{1-\nu}{1+\nu}
                    }
                }
            \right\}
        \right\rceil.
    \end{equation}
    Then the test \cref{eq:stop-rule-new} is satisfied within at most $2s_k+1$
    inner iterations. 
\end{itemize}
\end{proposition}

Combining \Cref{prop:iPPM-constant,proposition:inner-loop-new} shows that, for any fixed $\rho>0$ and $\epsilon_k=\epsilon/2$, the PBM variant in \Cref{alg:bundle-double-loop-varaint} finds an iterate satisfying $F(y_k)-F^\star\leq\epsilon$ within
\[
    \widetilde{\mathcal{O}}\left(
        \epsilon^{-\frac{2}{1+\nu}}
    \right),
    \qquad
    \nu\in[0,1],
\]
total iterations. The logarithmic factor arises from bounding the inner-loop complexity separately at each proximal center. As shown below, a refined analysis using \Cref{lemma:initial-proximal-gap} and careful counting across outer iterations can remove this factor.

\begin{theorem}[Overall complexity of the PBM variant]
\label{theorem:complexity-absolute-PBM}
Consider the PBM variant in \Cref{alg:bundle-double-loop-varaint}, and suppose that $f$ satisfies \cref{eq:weakly-smooth-intro} with $\nu\in[0,1]$ and $L_\nu>0$. For any $\epsilon\in(0,F(y_0)-F^\star)$, choose
$ \epsilon_k={\epsilon}/{2}, \ k\geq0.$
Define
$
    d_0:=\Dist(y_0,X^\star),
    \,
    \Phi_0:=F(y_0)-F^\star,
    \,
    R_0:={\Phi_0}/(\rho d_0^2).
$ 
Then the method generates an iterate satisfying $F(y_k)-F^\star\leq\epsilon$ within at most $T_{\mathrm{total}}$ inner iterations. 
\begin{itemize}
\setlength{\itemsep}{0pt}
    \item If $\nu=1$, then
    \begin{equation}
    \label{eq:total-complexity-absolute-smooth}
        T_{\mathrm{total}}
        \leq
        \left\lceil
        \frac{\rho d_0^2}{\epsilon}
    \right\rceil
        \left[
            3+
            \frac{
                2\log(1+L_1/\rho)
                +
                \log(3+2R_0)
            }{
                \log(1+\rho/L_1)
            }
        \right].
    \end{equation}

    \item If $\nu\in[0,1)$, then
    \begin{equation}
    \label{eq:total-complexity-absolute-holder}
        T_{\mathrm{total}}
        \leq
        \left\lceil
        \frac{\rho d_0^2}{\epsilon}
    \right\rceil
        \left[
            3
            +
            4\log(3+2R_0)
            +
            \frac{8}{1-\nu}\left(\frac{2}{1+\nu} \right)^{\frac{1-\nu}{1+\nu}}
            \frac{
                L_\nu^{\frac{2}{1+\nu}}
            }{
                \rho
                \epsilon^{\frac{1-\nu}{1+\nu}}
            }
        \right].
    \end{equation}
\end{itemize}
\end{theorem}

\begin{proof}
Let 
$
    N:=\lceil \frac{\rho d_0^2}{\epsilon}\rceil
$ be the outer iteration bound.  Choose $x^\star\in X^\star$ such that
$\|y_0-x^\star\|=d_0$. 
From \cref{eq:outer-complexity-absolute}, we know 
\[
    \min_{0\leq k\leq N-1}
    \bigl(F(y_{k+1})-F^\star\bigr)
    \leq
    \frac{\rho d_0^2}{2N}
    +
    \frac{\epsilon}{2}
    \leq
    \epsilon.
\]
Hence, the method requires at most $N$ outer iterations. The rest of the analysis is to carefully bound the number of inner-loop iterations using \Cref{proposition:inner-loop-new} and \Cref{lemma:initial-proximal-gap}. 

For each outer iteration $k$, recall the proximal gap 
$
    \Delta_k := F(y_k)-F_{1/\rho}(y_k),
$ and define the ratio $
    t_k
    :=
    {2\Delta_k}/{\epsilon}
$ for notational simplicity. 
Since $F_{1/\rho}(y_k)\geq F^\star$, we have
$
    \Delta_k
    \leq
    F(y_k)-F^\star.
$ 
Applying \cref{eq:average-inexatness} leads to
$
    \sum_{k=1}^{N-1}t_k \leq \frac{2}{\epsilon} \sum_{k=1}^{N-1}(F(y_k) - F^\star )
    \leq {\rho d_0^2}/{\epsilon} + N \leq 2N. 
$ 
By definition, we have 
$
    t_0     = 2{\Delta_0}/\epsilon  \leq
    {2\Phi_0}/{\epsilon}.
$ 
Consequently, we have 
$
    \sum_{k=0}^{N-1}t_k
    \leq
    {2\Phi_0}/{\epsilon}
    +2N. 
$ 
By concavity of the logarithm,
\begin{align}
\label{eq:absolute-proof-log-sum}
    \sum_{k=0}^{N-1}\log(1+t_k)
    \leq
    N\log\left(
        1+\frac{1}{N}\sum_{k=0}^{N-1}t_k
    \right)
    \leq
    N\log\left(
        3+\frac{2\Phi_0}{N\epsilon}
    \right)
    \leq
    N\log(3+2R_0),
\end{align}
where we used $N\epsilon\geq\rho d_0^2$.

\smallskip
\noindent
\textbf{Case 1: $\nu=1$.}
Consider the inner bound \cref{eq:inner-loop-smooth-new} in \Cref{proposition:inner-loop-new}, and we define
$
    H_k :=  \log_+\left(
        \frac{2L_1\delta_{0}}{\rho\epsilon}
    \right).
$ The estimate \cref{eq:initial-gap-unified} in \Cref{lemma:initial-proximal-gap} gives
$ \delta_{0} \leq \frac{L_1}{\rho}\Delta_k.$ 
Therefore, we have 
\[
    \frac{2L_1\delta_{0}}{\rho\epsilon}
    \leq
    ({L_1}/{\rho})^2t_k
    \leq
    (1+{L_1}/{\rho})^2(1+t_k),
\]
which implies
$
    H_k
    \leq
    2\log(1+{L_1}/{\rho})+\log(1+t_k).
$ 
By \Cref{proposition:inner-loop-new}, we have 
\[
    T_{\mathrm{inner},k}
    \leq
    2+
    \left\lceil
        \frac{H_k}{\log(1+\rho/L_1)}
    \right\rceil
    \leq
    3+
    \frac{H_k}{\log(1+\rho/L_1)}.
\]
Summing over $k=0,\ldots,N-1$ and applying
\cref{eq:absolute-proof-log-sum} yields
\[
    T_{\mathrm{total}} = \sum_{k=0}^{N-1} T_{\mathrm{inner},k}
    \leq
    N
    \left[
        3+
        \frac{
            2\log(1+L_1/\rho)
            +
            \log(3+2R_0)
        }{
            \log(1+\rho/L_1)
        }
    \right],
\]
which proves
\cref{eq:total-complexity-absolute-smooth}.

\smallskip
\noindent
\smallskip
\noindent
\textbf{Case 2: $\nu\in[0,1)$.}
Similar to \cref{eq:constant-for-proof}, define
\[
    r:=\frac{1-\nu}{1+\nu}\in(0,1],
    \qquad
    A:=
    \frac{
        L_\nu^{\frac{2}{1+\nu}}
    }{
        \rho({\epsilon}/{2})^r
    },
    \qquad
    B:=
    \frac{
        4(1+\nu)
    }{
        (1-\nu)\rho G_\nu({\epsilon}/{2})^r
    }.
\]
Here, $B$ denotes the last term in the inner-loop bound of
\Cref{proposition:inner-loop-new}. Similar to \cref{eq:constant-for-proof}, we have 
$B  \geq 4A.$  By \Cref{proposition:inner-loop-new}, the test \cref{eq:stop-rule-new} is
satisfied within at most $2s_k+1$ inner iterations, where
\[
    s_k
    :=
    \left\lceil
        \max\left\{
            1,\,
            2\log_+\left(
                \frac{\delta_0}{\epsilon/2}
            \right),\,
            B
        \right\}
    \right\rceil.
\]
Meanwhile, the bound \cref{eq:initial-gap-unified} in \Cref{lemma:initial-proximal-gap} gives
$
    \delta_0
    \leq
    \frac r2\Delta_k
    +
    \left(
        1+\frac r2
    \right)
    \frac{
        L_\nu^{\frac{2}{1+\nu}}
    }{
        \rho
    }
    \Delta_k^{1-r}.
$ 
Recalling that $t_k=2\Delta_k/\epsilon$, we obtain
\[
    \frac{\delta_0}{\epsilon/2}
    \leq
    \frac r2t_k
    +
    \left(
        1+\frac r2
    \right)
    At_k^{1-r}
    \leq
    (1+2A)(1+t_k),
\]
where the last inequality uses $r\in(0,1]$ and
$t_k^{1-r}\leq1+t_k$. Consequently,
\begin{equation}
\label{eq:holder-absolute-log-bound}
    \log_+\left(
        \frac{\delta_0}{\epsilon/2}
    \right)
    \leq
    \log(1+t_k)
    +
    \log(1+2A)
    \leq
    \log(1+t_k)+2A.
\end{equation}

From \cref{eq:holder-absolute-log-bound} and $B\geq4A$,
we know $
    s_k
    \leq
    1+2\log(1+t_k)+B.
$ 
Hence, we have 
$
    T_{\mathrm{inner},k}
    \leq
    2s_k+1
    \leq
    3+4\log(1+t_k)+2B.
$ 
Summing over $k=0,\ldots,N-1$ and using
\cref{eq:absolute-proof-log-sum} yields
\[
    T_{\mathrm{total}} = \sum_{k=0}^{N-1} T_{\mathrm{inner},k}\leq
    N
    \left[
        3
        +
        4\log(3+2R_0)
        + 2B
    \right].
\]
Substituting $B$ and the definition $G_\nu$ in \cref{eq:two-constants} leads to \Cref{eq:total-complexity-absolute-holder}. 
\end{proof}

\Cref{theorem:complexity-absolute-PBM} states that for any fixed proximal parameter $\rho>0$, \cref{eq:total-complexity-absolute-smooth} gives the overall iteration complexity $T_{\mathrm{total}} =  \mathcal O(\epsilon^{-1})$ for the smooth case $\nu = 1$. For $\nu\in[0,1)$, \cref{eq:total-complexity-absolute-holder} leads~to 
$
    T_{\mathrm{total}} = \mathcal O\left( L_\nu^{\frac{2}{1+\nu}}d_0^2\epsilon^{-\frac{2}{1+\nu}}
    \right),
$ 
and in particular, the nonsmooth case $\nu=0$ has the overall iteration complexity $ T_{\mathrm{total}} = \mathcal O\left(L_0^2d_0^2\epsilon^{-2}\right).$ All these rates match the order of those established in \Cref{theorem:complexity-classical-PBM,thm:classical-improved-explicit} for the classical PBM with $\rho > 0$ for the case $\nu = 1$ and $\rho=\Theta(\epsilon)$ for the case $\nu \in [0,1)$. Note that the PBM variant in \Cref{alg:bundle-double-loop-varaint} works for any fixed proximal parameter $\rho>0$ and the target accuracy $\epsilon$ enters in its bundle stopping criterion \cref{eq:stop-rule-new}.    

\begin{remark}[Alternative inexactness]
    In the PBM variant, we use an absolute inexactness criterion \cref{eq:stop-rule-new}. Other stopping criteria may also be considered. For example, in the smooth case, the bundle iterations can be terminated using a relative inexactness condition of the form $e_j \leq \frac{\rho}{2}\|z_{j+1} - y_k\|^2$. This type of relative inexactness has been used to design accelerated inexact PPM in \cite{monteiro2013accelerated}. 
\end{remark}

\subsection{Comparison with modern proximal bundle design} \label{subsection:comparison-modern-PBM}

Our PBM variant in \Cref{alg:bundle-double-loop-varaint} is closely related to several recent PBM developments \cite{liang2021proximal,liang2024unified,liang2026proximal,liang2025primal,
fersztand2024modified}, all of which replace the classical descent test \cref{eq:test-PBM} with alternative criteria for updating the proximal center in the outer loop.

Liang and Monteiro \cite{liang2021proximal} proposed a relaxed proximal bundle method that replaces the classical descent test \cref{eq:test-PBM} with a different criterion. In our notation, this criterion takes the form
\begin{equation} \label{eq:alterantive-criterion-1}
 t_j^{\mathrm{prox}} := F(\tilde y_j)
+\frac{\rho}{2}\|\tilde y_j-y_k\|^2
-
\left(
F_j(z_{j+1})
+\frac{\rho}{2}\|z_{j+1}-y_k\|^2
\right) \leq \frac{\epsilon}{2},
\end{equation}
where $\tilde y_j$ is the best bundle iterate in terms of the proximal subproblem $F(\cdot)+\frac{\rho}{2}\|\cdot-y_k\|^2$, and $\epsilon$ is the target accuracy. 
This stopping criterion \cref{eq:alterantive-criterion-1} was also used later in \cite{liang2025primal,liang2026proximal}. 
For nonsmooth convex functions with bounded subgradients, \cite{liang2021proximal} establishes an overall complexity of $\bigO(\epsilon^{-2})$, matching the rate of our PBM variant when $\nu=0$. Its analysis, however, does not exploit H{\"o}lder smoothness and therefore provides no improved complexity guarantees for $\nu\in(0,1]$. 
The work \cite{liang2025primal} reveals a duality between the conditional gradient method and the cutting-plane scheme used within PBMs, but does not exploit H{\"o}lder smoothness either. The work \cite{liang2026proximal} considers a PBM variant with \cref{eq:alterantive-criterion-1} for H{\"o}lder-smooth functions and establishes the complexity $\widetilde{\bigO}\big(\epsilon^{-\frac{2}{1+\nu}}\big)$, where $\widetilde{\bigO}$ hides a logarithmic factor. Its inner bundle routine, however, requires the full cutting-plane model. In contrast, our analysis accommodates the more general model in \cref{assump:bm-single} and removes the logarithmic factor.

The work \cite{liang2024unified} introduces another stopping rule
\begin{equation} \label{eq:alterantive-criterion-2}
 t_j^{\mathrm{obj}} :=F(\hat y_j)
-
\left(
F_j(z_{j+1})
+\frac{\rho}{2}\|z_{j+1}-y_k\|^2
\right) \leq \frac{\epsilon}{2},
\end{equation}
where $\hat y_j$ is the best bundle iterate in terms of the objective $F$. It considers the hybrid regularity condition
$
\|\nabla f(x)-\nabla f(y)\|
\leq
2M+L\|x-y\|,
\; \forall x,y\in\RR^n,
$ 
where $M,L\geq 0$, which recovers the bounded subgradient setting when $L=0$ and the $L$-smooth setting when $M=0$. \cite[Proposition~2.1]{liang2024unified} reduces the H{\"o}lder condition \cref{eq:weakly-smooth-intro} to the hybrid condition through a tolerance-dependent choice of $(M,L)$ and derives the complexity $\widetilde{\bigO}\big(\epsilon^{-\frac{2}{1+\nu}}\big)$.  More closely related to our stopping criterion \cref{eq:stop-rule-new}, the work \cite{fersztand2024modified} uses the model error $ F(z_{j+1})-F_j(z_{j+1})$ to determine whether to update the outer iterate. It interprets the null steps of the resulting bundle method as fully corrective Frank--Wolfe iterations applied to a dual problem. Its analysis, however, is restricted to objectives consisting of the sum of a smooth convex function and a convex piecewise-linear function.

The stopping measures \cref{eq:alterantive-criterion-1,eq:alterantive-criterion-2} are closely related to our model-error test \cref{eq:stop-rule-new}. Since $\tilde y_j$ and $\hat y_j$ are chosen as the best bundle iterates under their respective criteria, we have
$    t_j^{\mathrm{prox}} \leq e_j,\,
    t_j^{\mathrm{obj}}\leq e_j-\frac{\rho}{2}\|z_{j+1}-y_k\|^2\leq e_j.
$ 
Thus, our model-error test is stronger at the same tolerance. One main advantage of \cref{eq:stop-rule-new} is that the optimality of the bundle subproblem directly yields
$
    z_{j+1}=y_k-\frac{1}{\rho}s_{j+1},\,
    s_{j+1}\in\partial_{e_j}F(z_{j+1}),
$ 
so it immediately certifies the subgradient inexactness; see \cref{eq:z-inexact-update}. This fact leads to the clean outer-loop complexity in \cref{eq:outer-complexity-absolute}. In contrast, $t_j^{\mathrm{prox}}$ certifies proximal accuracy at an auxiliary best point, and $t_j^{\mathrm{obj}}$ is tailored to objective-based outer progress; neither directly controls the subgradient inexactness at the accepted candidate.

\section{Numerical experiments}
\label{sec:numerics}
In this section, we present numerical experiments to demonstrate the convergence behavior of the PBMs under H\"older smoothness. In our implementation of \ProDes{}, given an iterate $y_k$, we use the linearization for the first model $
    f_0(y) =f(y_k) + \innerproduct{\nabla f(y_k)}{y-y_k},
$ and the essential two-cut model 
$
    f_j(y) = \max\{ f(z_{j}) + \innerproduct{g_{j}}{y-z_{j}},f_{j-1}(z_{j}) + \innerproduct{a_{j}}{y-z_{j}} \}
$ 
for $j > 0$ that do not satisfy the condition $\mathcal{S}$; see \cref{eq:bundle-subgradient,eq:bundle-aggregation}.

\subsection{Convergence of \ProDes{}}
\label{subsec:numeric-inner-loop}
We first illustrate the convergence rate of the subroutine \ProDes{} established in \cref{lemma:error,coro:convergence-d}. Consider the following family of functions
\begin{align}
    \label{eq:familiy-function}
    f_{\nu}(x) = \frac{1}{m}
    \sum_{i=1}^{m} \frac{1}{1+\nu} |a_i^\tr x - b_i |^{1+\nu},
\end{align}
where $\nu \in [0,1]$, and $a_i \in \RR^{n}$, $n = 5$, and $m = 20$ are problem data. The quantity $\nu$ determines the smoothness level, i.e., the function $f_{\nu}$ satisfies \cref{eq:weakly-smooth-intro} with the smoothness level $\nu$. The data $a_i, b_i$, $i = 1,\ldots,m$, are generated as follows: $a_i \sim \mathcal{N}(0,I_n)$ (a normal distribution with mean $0$ and covariance matrix $I_n$, where $I_n$ is an identity matrix with dimension $n$), and $b_i \sim \mathcal{N}(0,0.01)$. In the experiment, we vary $\nu \in \{0,1/3,2/3,1\}$, fix a center point $\bar x = 10^{-3}\times v \in \RR^n $, where $v\sim \mathcal{N}(0,I_n)$, choose a proximal parameter $\rho = 0.5$, and run \ProDes{}($\bar x, \rho, \mathcal{S}$) in \Cref{alg:Proxi-descent-subproblem}. The numerical result is presented in \cref{fig:numerical-experiment-inner-loop}. The left figure shows the evolution of the value gap $\delta_j$, and the right figure shows the evolution of the best model error $\min_{1\leq k \leq j} e_k$. To compute the true Moreau envelope $F_{1/\rho}(\bar x)$ in the definition of $\delta_j =  F_{1/\rho}(\bar x)- m_j$, we run gradient descent on $\min_{x} f_{\nu}(x) + \frac{\rho}{2}\|x - \bar x \|^2$ with Armijo backtracking line search for $20,000$ iterations or until the gradient norm is less than $10^{-12}$. As established in \cref{lemma:error,coro:convergence-d}, both quantities have the convergence rate of $\bigO(j^{-\frac{1+\nu}{1-\nu}})$ for $\nu \in [0,1)$ and the $\bigO(e^{-aj})$ with some $a \in (0,1)$ for $\nu = 1$. This is consistent with \cref{fig:numerical-experiment-inner-loop} for different values of $\nu$. In particular, as the smoothness level increases, the subroutine \ProDes{} converges faster. We emphasize that the convergence improvement is automatic with no parameter tuning in \Cref{alg:Proxi-descent-subproblem}. 

\begin{figure}[t]
    \centering
    \begin{subfigure}[t]{0.45\textwidth}
        \centering
        \includegraphics[width=\textwidth]{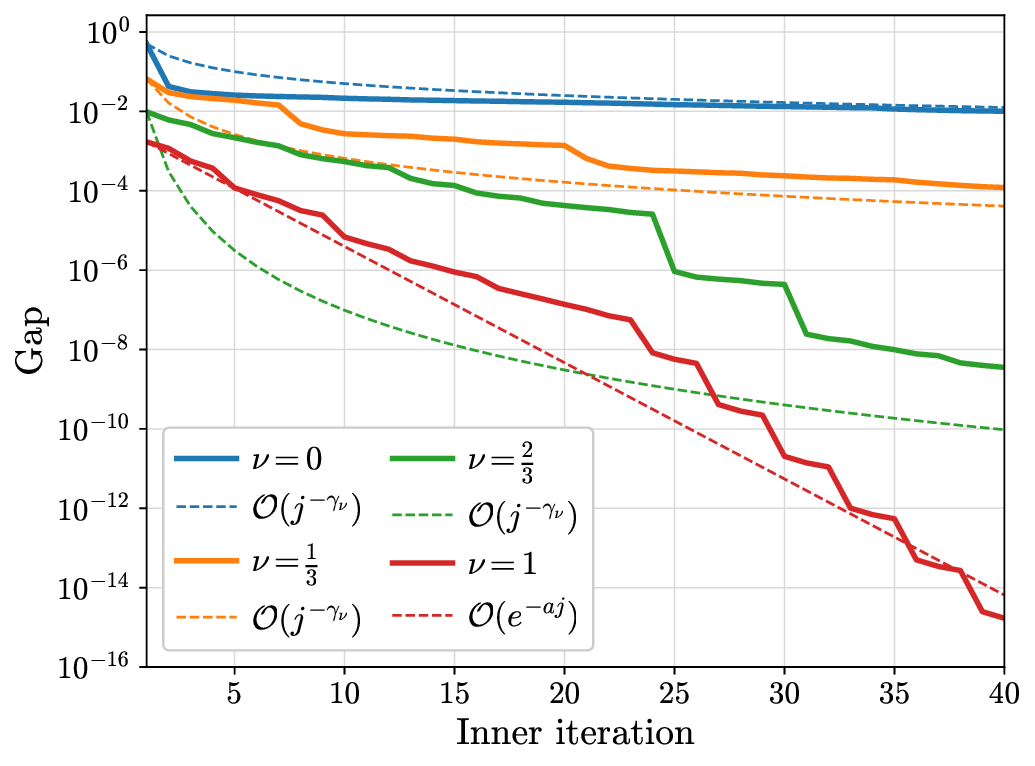}
    \end{subfigure}
    \hspace{3mm}
    \begin{subfigure}[t]{0.45\textwidth}
        \centering
        \includegraphics[width=\textwidth]{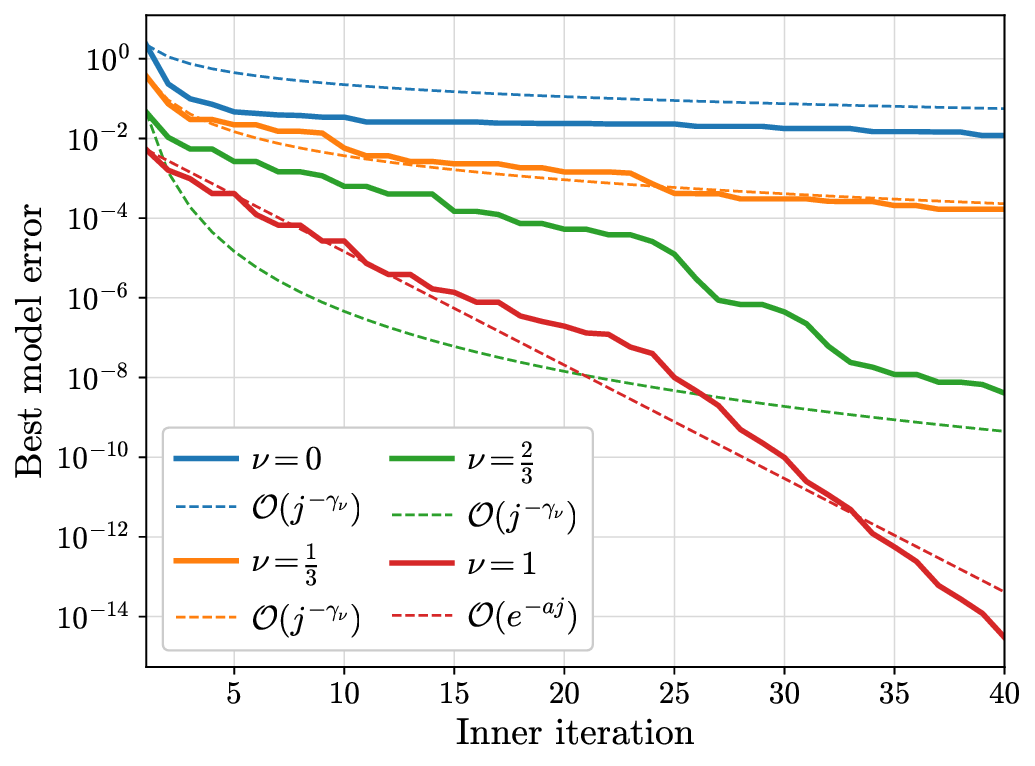}
    \end{subfigure}

\caption{Numerical behavior of \ProDes{} for solving \cref{eq:familiy-function} with $\nu \in \{0,1/3,2/3,1\}$. The gap and best model error denote $\delta_j$ and $\min_{1\leq k \leq j}e_k$, respectively, defined in \cref{eq:errors-j-th-iteration-d,eq:errors-j-th-iteration-e}. We also denote $\gamma_{\nu} = \frac{1+\nu}{1-\nu}$ for $\nu \in [0,1)$. The geometric decay constant $a$ is approximately $0.67$ for the case $\nu = 1$.} 
\label{fig:numerical-experiment-inner-loop}
\end{figure}

\subsection{Sparse logistic regression}
\label{subsec:logistic}
We next consider the sparse binary logistic regression
\begin{equation}
    \label{eq:sparse-logistic-regression}
    \min_{x \in \RR^n}     \frac{1}{m}
    \sum_{i=1}^{m}
    \log\!\left(1+\exp\!\left(-b_i a_i^{\tr}x\right)\right)
    +
    \lambda \|x\|_1,
\end{equation}
where  $ a_i \in \RR^n, b_i \in \{-1,1\},i=1,\ldots,m,$ and $\lambda > 0$ is the regularization parameter.  The nonsmooth term $\lambda \|x\|_1$ is standard to promote sparsity in the solution. We view \cref{eq:sparse-logistic-regression} in the form of \cref{eq:pb-main} with $f(x) =\frac{1}{m}
    \sum_{i=1}^{m}
    \log\!\left(1+\exp\!\left(-b_i a_i^{\tr}x\right)\right) $ and $h(x) = \lambda \|x\|_1$, consider two problem datasets ``a9a'' ($ m = 32,561, n =  123$) and ``gisette'' ($m = 6000, n = 5000$) from the library LIBSVM \cite{CC01a}, and set $\lambda = 0.1\|A^\tr b\|_{\infty}/(2m)$, where $A = [a_1^\tr; a_2^\tr; \dots a_m^\tr]$. 
    
    We fix the proximal parameter $\rho = 10^{-2}$, run the classical PBM \Cref{alg:bundle-double-loop} with $\beta \in \{ 0.25,0.75\}$ and the PBM variant \Cref{alg:bundle-double-loop-varaint} with $\epsilon_k \in  \{10^{-2}, 10^{-3}\}$ for $10^3$ iterations, and report the numerical result in \cref{fig:numerical-experiment-sparse-logistic-regression} where each flat segment corresponds to the cycle of each inner loop. In this application, the subproblem does not admit a closed-form solution, but it can be solved by bisection. We provide the implementation details in \cref{sec:implementation}. 
    The true optimal value $F^\star$ is estimated by running FISTA \cite{beck2009fast} until the proximal-gradient residual $
R_k
=
L_k\left\|
y_k-
\operatorname{soft}\!\left(
y_k-\frac{\nabla f(y_k)}{L_k},
\frac{\lambda}{L_k}
\right)
\right\|
$
satisfies
$
R_k
\leq
\delta \max\left\{1,\|y_{k+1}\|\right\},
$
where $1/L_k$ is the stepsize in FISTA and $\delta=10^{-8}$ for ``a9a'' and  $\delta=10^{-7}$ for ``gisette''. We observe that the classical PBM with both $\beta = 0.25$ and $0.75$ quickly achieves high accuracy on the two considered datasets. In particular, it converges to the accuracy of $10^{-13}$ at about $800$ iterations. On the other hand, the PBM variant stops making progress when reaching the accuracy of a certain level. For example, for the dataset ``gisette", the PBM variant with $\epsilon = 10^{-2}$ (resp. $\epsilon = 10^{-3}$) stops making progress when the accuracy is around $10^{-3}$ (resp. $10^{-6}$). This behavior is qualitatively consistent with the use of a fixed
absolute inner tolerance. \cref{theorem:complexity-absolute-PBM} only guarantees convergence to an
$\epsilon$-optimal solution but does not guarantee further progress.

\begin{figure}[t]
    \centering
    \begin{subfigure}[t]{0.45\textwidth}
        \centering
        \includegraphics[width=\textwidth]{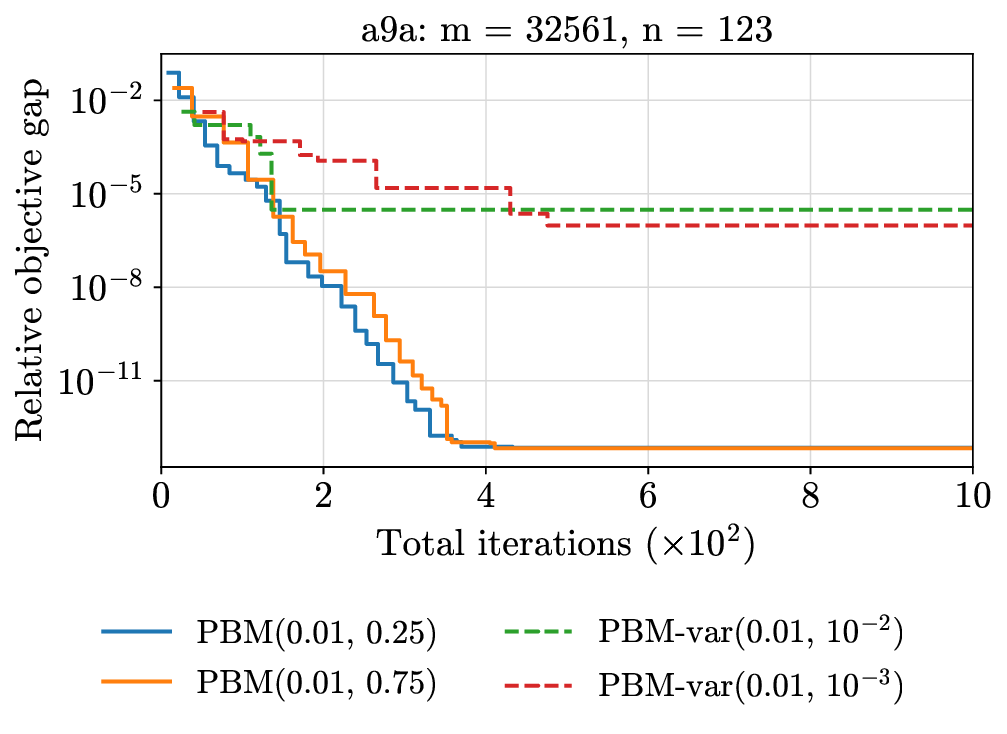}
    \end{subfigure}
    \hspace{3mm}
    \begin{subfigure}[t]{0.45\textwidth}
        \centering
        \includegraphics[width=\textwidth]{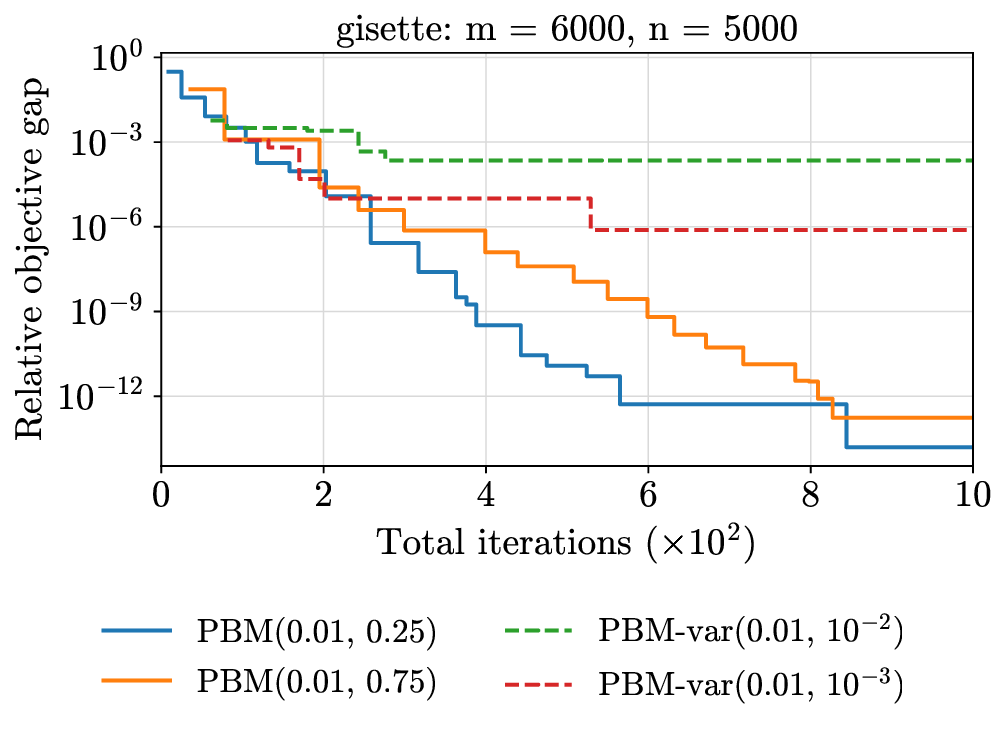}
    \end{subfigure}
\caption{Comparison between the classical PBM \Cref{alg:bundle-double-loop} and the PBM variant \Cref{alg:bundle-double-loop-varaint} on solving the sparse logistic regression \cref{eq:sparse-logistic-regression} for the datasets ``a9a" and ``gisette". The notation PBM($x,y$) denotes the classical PBM with parameters $\rho = x$ and $\beta = y$, and PBM-var($x,y$) denotes the PBM variant with parameters $\rho = x$ and $\epsilon_k = y$ for all $k$. Relative objective gap denotes the quantity $|\frac{F(x_k) - F^\star}{F^\star}|$. For the PBM variant, we report the best relative objective gap.
}
\label{fig:numerical-experiment-sparse-logistic-regression}
\end{figure}

\subsection{Hinge-loss classification}
\label{subsec:hinge-loss} 
As our final experiment, we consider the hinge-loss classification in the form of 
\begin{align}
    \label{eq:hinge-loss}
    \min_{x} \frac1m\sum_{i=1}^m\max\{0,1-b_i a_i^\tr x\} + \frac{\lambda}{2}\lVert x\rVert^2
\end{align}
where  $ a_i \in \RR^n, b_i \in \{-1,1\},i=1,\ldots,m,$ are problem data and $\lambda > 0$ is the regularization parameter. We treat $f(x) = \frac1m\sum_{i=1}^m\max\{0,1-b_i a_i^T x\}$ and $h(x) = \frac{\lambda}{2}\lVert x\rVert^2$. It is clear that the function $f$ is a nonsmooth function satisfying \cref{eq:weakly-smooth-intro} with $\nu = 0$. Similarly, we consider two problem datasets ``a9a'' ($ m = 32,561, n =  123$) and ``w8a'' ($m = 49749, n = 300$) from the library LIBSVM \cite{CC01a}, and set $\lambda = 10^{-3}$. We fix the proximal parameter $\rho = 10^{-3}$, run the classical PBM \Cref{alg:bundle-double-loop} with $\beta \in \{ 0.25,0.75\}$ and the PBM variant \Cref{alg:bundle-double-loop-varaint} with $\epsilon_k = \epsilon \in  \{10^{-2}, 10^{-4}\}$ for $10^4$ iterations. The numerical result is reported in \cref{fig:numerical-experiment-hinge-loss}. The true optimal value $F^\star$ is estimated by solving the dual problem of \cref{eq:hinge-loss} (see \cref{subsec:hinge-loss-Fstar} for details). Similar to \cref{fig:numerical-experiment-sparse-logistic-regression}, the classical PBM is able to continue decreasing the cost value, while the PBM variant plateaus after reaching a certain accuracy depending on the inner-loop inexactness. 

\begin{figure}[t]
    \centering
    \begin{subfigure}[t]{0.45\textwidth}
        \centering
        \includegraphics[width=\textwidth]{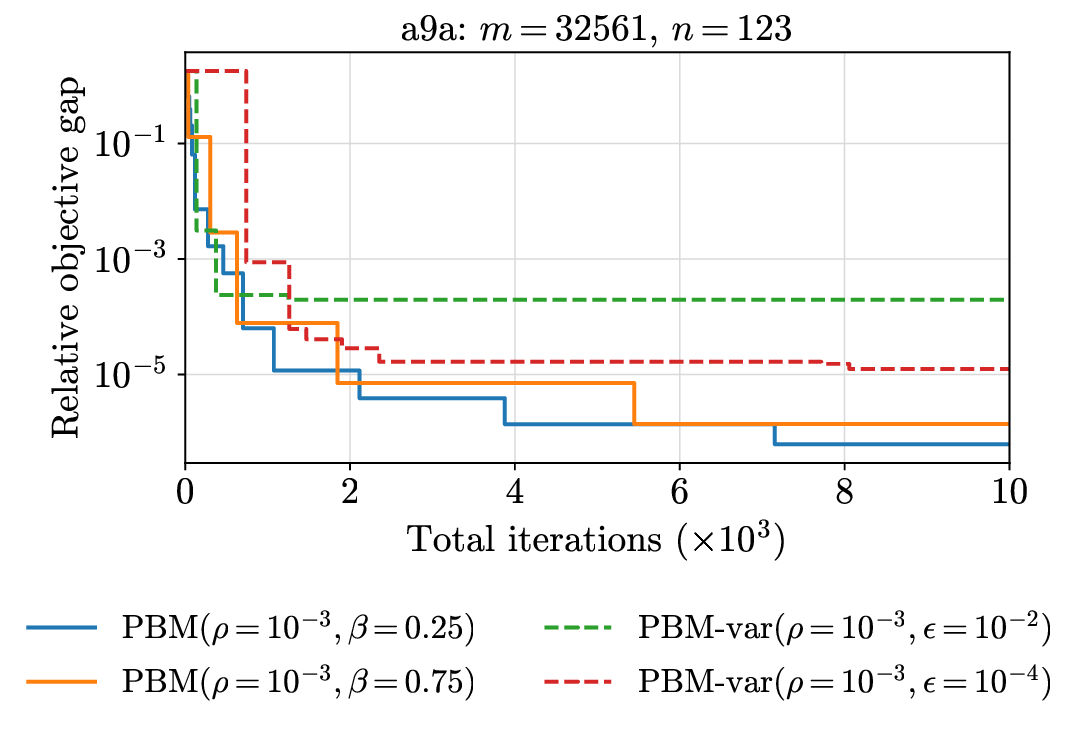}
    \end{subfigure}
    \hspace{3mm}
    \begin{subfigure}[t]{0.45\textwidth}
        \centering
        \includegraphics[width=\textwidth]{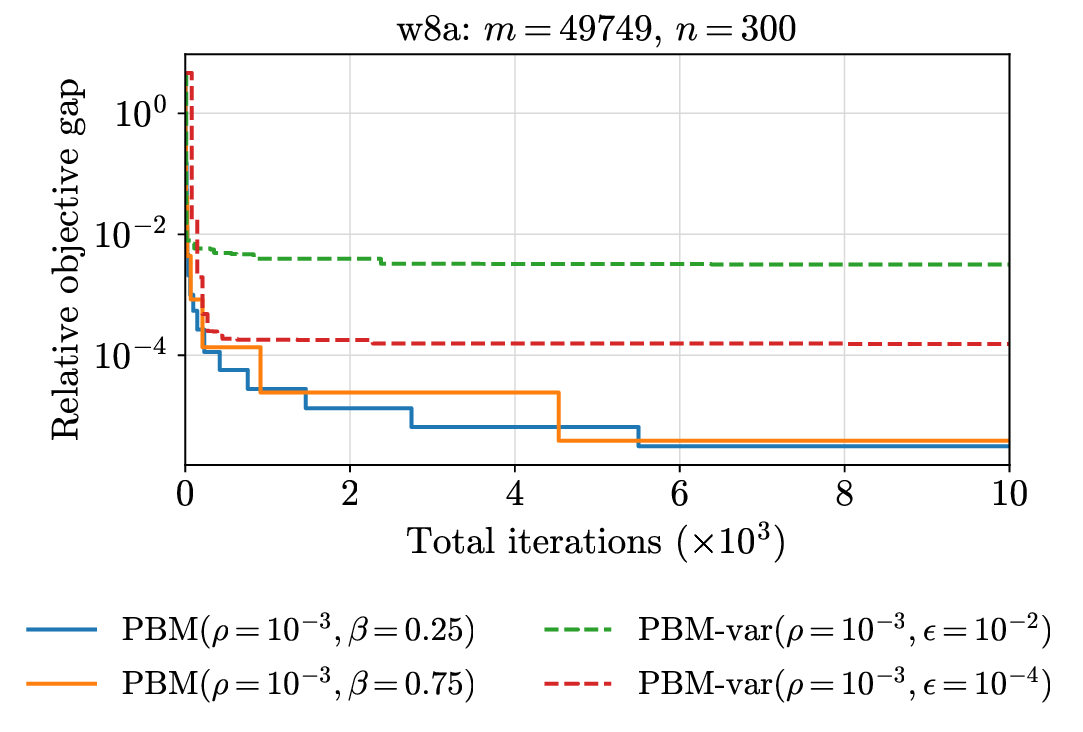}
    \end{subfigure}
\caption{Comparison between the classical PBM \Cref{alg:bundle-double-loop} and the PBM variant \Cref{alg:bundle-double-loop-varaint} on solving  \cref{eq:hinge-loss} for the datasets ``a9a" and ``w8a". The notation and reported quantities are the same as in \cref{fig:numerical-experiment-sparse-logistic-regression}.}
\label{fig:numerical-experiment-hinge-loss}
\end{figure}

\section{Conclusion}
\label{sec:conclusion}
This paper has revisited the classical PBM from a double-loop perspective, by viewing the null-step cycle as an inner solver for the proximal subproblem and the descent test as its stopping criterion. We show that the convergence rate of the inner bundle iterations automatically adapts to H{\"o}lder smoothness. Combining this inner-loop rate with the descent-step analysis yields improved complexity bounds for the classical PBM with sharper parameter dependence. We have also studied a computable model-error test that leads to a PBM variant, which has a clear inexact proximal-point interpretation. Overall, our framework separates the roles of the descent and null steps and provides a unified understanding of PBMs under different stopping rules. Future directions include developing universally accelerated PBMs through suitable acceleration schemes and extending the double-loop perspective to stochastic optimization.

\subsection*{Use of generative AI}
After finishing the first version of this work, the authors used ChatGPT 5.6 Sol and Claude Opus 5 to assist with grammar checking and to check and improve intermediate derivations and proof arguments. The authors independently verified all mathematical statements and proofs. The~authors assume responsibility for all content.

\bibliographystyle{unsrt}
\bibliography{references}

\newpage
\appendix
\section*{Appendix}

\section{Sequence analysis} \label{app:sequence}
One key ingredient in our non-asymptotic analysis is the sequence convergence in \cref{lemma:sequence-recursion}. For convenience, we re-state this result below. 

\begin{lemma}
\label{lemma:sequence-recursion-app}
Let $p \in [1,2]$, $c>0$, and let $\{a_j\}_{j\geq 0}$ be a nonnegative sequence with $a_0 >0$ satisfying
\begin{equation}
\label{eq:sequence-recursion}
    a_j \leq a_{j-1} -  c a_j^p,
    \qquad j\geq 1.
\end{equation}
If $p=1$, then $a_j\leq {a_0}{(1+c)^{-j}}, \, j\geq 1.$
    If $p\in(1,2]$, then we have 
    \begin{equation*}
        a_j \leq \max\left\{ a_0 e^{-j/2},\left(\frac{2}{(p-1)cj}\right)^{1/(p-1)}
        \right\}, \qquad j\geq 1.
    \end{equation*}
\end{lemma}

\begin{proof}
If $p=1$, \cref{eq:sequence-recursion} gives
$
    (1+c)a_j\leq a_{j-1},
$
and iterating this inequality yields
$
    a_j\leq a_0(1+c)^{-j}.
$ 

Now suppose $p\in(1,2]$ and define $r:=p-1\in(0,1]$. Since
$\{a_j\}$ is nonincreasing, the result is immediate if $a_j=0$.
We thus assume $a_j>0$, which also implies $a_i>0$ for 
$i=0,\ldots,j$. By \cref{eq:sequence-recursion},
\begin{align}
    \label{eq:sequence-recursion-reform}
    \frac{a_{i-1}}{a_i}
    \geq
    1+ca_i^r,
    \qquad i=1,\ldots,j.
\end{align}
Writing $t_i:=ca_i^r$, we obtain
\begin{align*}
    \frac{a_i^{-r}-a_{i-1}^{-r}}{rc} =\frac{a_i^{-r}}{rc}
    \left[
        1-
        \left(\frac{a_i}{a_{i-1}}\right)^r
    \right] \geq
    \frac{1-(1+t_i)^{-r}}{rt_i},
\end{align*}
where the inequality uses \cref{eq:sequence-recursion-reform}. Next, for every $t\geq 0$ and $r\in(0,1]$, we claim that
\begin{align}
    \label{eq:cliam-lower}
       1-(1+t)^{-r}
    \geq
    \frac{rt}{1+t}.
\end{align}
It follows that 
\begin{equation}
\label{eq:sequence-inverse-power-increment}
    \frac{a_i^{-r}-a_{i-1}^{-r}}{rc}
    \geq
    \frac{1}{1+t_i}.
\end{equation}
Moreover, taking the log of \cref{eq:sequence-recursion-reform} gives
\begin{equation}
\label{eq:sequence-logarithmic-increment}
    \log\left(\frac{a_{i-1}}{a_i}\right)
    \geq
    \log(1+t_i)
    \geq
    \frac{t_i}{1+t_i},
\end{equation}
where the last inequality follows from the fact that 
$
    \log(1+t)  = \int_0^t \frac{1}{1+s}\,ds \geq \int_0^t \frac{1}{1+t}\,ds =  \frac{t}{1+t},
    \, t\geq 0.
$ 
Adding \cref{eq:sequence-inverse-power-increment,eq:sequence-logarithmic-increment}
and summing over $i=1,\ldots,j$ gives
\[
    \frac{a_j^{-r}-a_0^{-r}}{rc}
    +
    \log\left(\frac{a_0}{a_j}\right)
    \geq j.
\]
Therefore, at least one of the following inequalities must hold: $
    {(a_j^{-r}-a_0^{-r})}/({rc})
    \geq
    {j}/{2}$,
   or $ 
    \log\left({a_0}/{a_j}\right)
    \geq
    {j}/{2}.
$ 
The first inequality implies
\[
    a_j
    \leq
    \left(a_0^{-r}+\frac{rcj}{2}\right)^{-1/r}
    \leq
    \left(\frac{2}{rcj}\right)^{1/r},
\]
and the second gives $ a_j\leq a_0e^{-j/2}$. Combining these two cases proves the desired bound.

\textbf{Proof of \cref{eq:cliam-lower}}: The function $t\mapsto g(t)= (1+t)^{1-r}$ is concave if $r \in (0,1]$. Thus, it holds that $(1+t)^{1-r} = g(t) \leq g(0) + g^\prime(0) t  = 1 + t -rt.$ Dividing both sides by $(1+t)$ and simplifying the expression gives \cref{eq:cliam-lower}.
\end{proof}

\section{Proof of \cref{proposition:outer-loop-classical}}
\label{appendix:outer-loop-PBM}
    Recall that $\Phi_k =F(y_k)-F^\star$ denotes the cost gap and $ \sup \{\Dist(x,X^\star) \mid F(x) \leq F(y_0)\} \leq D  < +\infty$. 
    Note that \cref{lemma:proximal-gap} implies \cref{eq:proximal-gap-objective-gap}, which is rewritten below for convenience 
    \begin{equation*}
    \label{eq:claim}
\begin{aligned}
\Delta_k & \ge
\begin{cases}
\dfrac{1}{2\rho}
\left(
\dfrac{\Phi_k}
{D}
\right)^2,
&
\text{if } \Phi_k
\le
\rho D^2,
\\
\dfrac{1}{2}\Phi_k,
&
\text{otherwise}.
\end{cases}
\end{aligned}
\end{equation*}

\textbf{Case 1}: $\Phi_k > \rho D^2$.  
    From \cref{lemma:outer-loop-PBM-classical}, the cost improvement is 
$            \Phi_{k+1} \leq \Phi_k - \beta \Delta_k \leq \Phi_k - \frac{\beta}{2} \Phi_k = \left (1-\frac{\beta}{2} \right)\Phi_k $. 
        Thus, there is at most $\frac{\log(\Phi_0/(\rho D^2))}{\log(1 / (1-{\beta}/{2}))}$ iterations such that  $\Phi_k >  \rho D^2$.  
        
\textbf{Case 2}:  $\Phi_k\leq \rho D^2$. 
        From \cref{lemma:outer-loop-PBM-classical}, the cost improvement is 
        $
             \Phi_{k+1} \leq \Phi_k - \beta \Delta_k \leq \Phi_k  - \frac{\beta}{2\rho D^2} \Phi_k^2.
       $ 
        Let $c = \frac{\beta}{2\rho D^2}.$ Since $\Phi_k - \Phi_{k+1} \geq c \Phi_k^2$, it follows that 
        \[
            \begin{aligned}
            \frac{1}{\Phi_{k+1}}-\frac{1}{\Phi_k}
            &=
            \frac{\Phi_k-\Phi_{k+1}}
            {\Phi_k\Phi_{k+1}} \ge
            \frac{c\Phi_k^2}
            {\Phi_k\Phi_{k+1}} =
            c\frac{\Phi_k}{\Phi_{k+1}} \ge c.
            \end{aligned}
        \]
Let $k_0 = \min \{ k  : \Phi_k \leq \rho D^2 \}$ be the first index such that $\Phi_k$ is less than $\rho D^2$. 
Telescoping the above inequality gives
$
\frac{1}{\Phi_{k_0+t}}
\ge
\frac{1}{\Phi_{k_0}}
+ ct, 
$ 
which further implies
\[
\Phi_{k_0 + t}
\le
\frac{1}{
\frac{1}{\Phi_{k_0}}
+
ct} = \frac{\Phi_{k_0}}{1+\Phi_{k_0}ct} \leq \frac{1}{ct} 
=\frac{2\rho D^2}{\beta t}, \quad \forall t \geq 1. 
\]
Thus, there are at most $\frac{2\rho D^2}{\beta \epsilon}$ iterations before reaching $\Phi_k \leq \epsilon$. 

Summing the bounds in the above two cases finishes the proof.

\section{Proof of \Cref{lemma:initial-proximal-gap}} \label{appendix:initial-proximal-gap}
Since $y_k\notin X^\star$, we have $\Delta_k>0$. For notational simplicity, define
\[
    \lambda_k
    :=
    L_\nu^{\frac{2}{1+\nu}}
    \Delta_k^{-\frac{1-\nu}{1+\nu}},
    \qquad
    \eta_\nu
    :=
    \frac{1-\nu}{2(1+\nu)}.
\]
Applying the weighted arithmetic--geometric mean inequality
\(a^\theta b^{1-\theta}\leq\theta a+(1-\theta)b\) with
$
    a=\frac{\lambda_k}{1+\nu}\|u\|^2, b=\frac{\Delta_k}{1+\nu}, \theta = \frac{1+\nu}{2},
$  
gives that for all $u$, 

\[
\begin{aligned}
\frac{L_\nu}{1+\nu}\|u\|^{1+\nu}
&=
\left(
\frac{\lambda_k}{1+\nu}\|u\|^2
\right)^{\frac{1+\nu}{2}}
\left(
\frac{\Delta_k}{1+\nu}
\right)^{\frac{1-\nu}{2}}
\leq
\frac{\lambda_k}{2}\|u\|^2+\eta_\nu\Delta_k.
\end{aligned}
\]
Consequently, the upper bound \cref{eq:weakly-smooth-conseq} from the H{\"o}lder smoothness implies
\begin{equation}
\label{eq:holder-quadratic-majorization}
    f(y_k+u)
    \leq
    f(y_k)
    +
    \langle g_k,u\rangle
    +
    \frac{\lambda_k}{2}\|u\|^2
    +
    \eta_\nu\Delta_k,
\end{equation}
where $g_k$ is the subgradient used to construct the initial cutting
plane at $y_k$.

Let $z_1$ be the first trial point and write $d=z_1-y_k$. Since the
initial model contains this cutting plane, we have $
    f(y_k)+\langle g_k,d\rangle
    \leq
    f_0(z_1).
$ 
Applying \cref{eq:holder-quadratic-majorization} with $u=td$ and
using convexity of $h$, we obtain, for every $t\in[0,1]$,
\[
\begin{aligned}
    F(y_k+td)
    &\leq
    (1-t)F(y_k)
    +
    tF_0(z_1)
    +
    \frac{\lambda_k t^2}{2}\|d\|^2
    +
    \eta_\nu\Delta_k.
\end{aligned}
\]
The definition of the Moreau envelope therefore yields
\[
\begin{aligned}
    F(y_k) - \Delta_k = F_{1/\rho}(y_k)
    &\leq
    F(y_k+td)
    +
    \frac{\rho t^2}{2}\|d\|^2 \\
    &\leq
    (1-t)F(y_k)
    +
    tF_0(z_1)
    +
    \frac{(\lambda_k+\rho)t^2}{2}\|d\|^2
    +
    \eta_\nu\Delta_k.
\end{aligned}
\]
Moreover, the definitions of $\Delta_k$ and $\delta_0$ give
$
    F_0(z_1) = F(y_k) - \Delta_k  - \delta_0 - \frac{\rho}{2}\|d\|^2.
$ 
Substituting this identity into the inequality above shows that
\[
    \Delta_k
    \geq
    t(\Delta_k+\delta_0)
    +
    \frac{
        \rho t-(\lambda_k+\rho)t^2
    }{2}\|d\|^2
    -
    \eta_\nu\Delta_k.
\]
Choosing $t=\rho/(\lambda_k+\rho)$, we obtain
 $
    (1+\eta_\nu)\Delta_k
    \geq
    \frac{\rho}{\lambda_k+\rho}
    (\Delta_k+\delta_0), 
$
which implies 
\[
    \delta_0
    \leq
    \eta_\nu\Delta_k
    +
    (1+\eta_\nu)
    \frac{\lambda_k}{\rho}\Delta_k.
\]
Substituting the definitions of $\lambda_k$ and $\eta_\nu$ gives
\cref{eq:initial-gap-unified}.

\section{Implementation details}
\subsection{Solving the composite subproblem}
\label{sec:implementation}
We detail the implementation for solving the composite model subproblem
\[
\min_x\;
f_j(x) +\lambda\lVert x\rVert_1
+\frac{\rho}{2}\lVert x-y_k\rVert^2,
\]
where $f_j(x)  =  \max\{ f(z_{j}) + \innerproduct{g_{j}}{x-z_{j}},f_{j-1}(z_{j}) + \innerproduct{a_{j}}{x-z_{j}} \}$. Although the subproblem does not have a closed-form solution, it reduces to a one-dimensional concave dual problem. Specifically, let $\ell_1(x) = f(z_{j}) + \innerproduct{g_{j}}{x-z_{j}}$ and $ \ell_2(x) = f_{j-1}(z_{j}) + \innerproduct{a_{j}}{x-z_{j}}$. Note that the max of two affine functions can be written equivalently as
\[
\max\{\ell_1(x),\ell_2(x)\}
=
\max_{0\leq\theta\leq1}
\left[
(1-\theta)\ell_1(x)+\theta\ell_2(x)
\right]
.
\]
Let $\Psi(x,\theta) = \left[
(1-\theta)\ell_1(x)+\theta\ell_2(x)
\right] +\lambda\lVert x\rVert_1
+\frac{\rho}{2}\lVert x-y_k\rVert^2$ and $q(\theta) = \min_{x} \Psi(x,\theta) $. Since $\Psi$ is convex in $x$ and concave in $\theta$ and the set $[0,1]$ is compact and convex, by Sion's theorem, the subproblem becomes 
\begin{align*}
    \min_x\;
\max_{0\leq\theta\leq1}
\Psi(x,\theta)  =  \max_{0\leq\theta\leq1} \min_x\;
\Psi(x,\theta).
\end{align*}
For a fixed $\theta$, the inner minimization has the solution
\[
x(\theta)
=
\operatorname{soft}\!\left(
y_k-\frac{g(\theta)}{\rho},
\frac{\lambda}{\rho}
\right),
\]
where $g(\theta) = (1-\theta)g_j+\theta a_j$. To get the optimal $\theta$, we only need to solve
$
 \max_{0\leq\theta\leq1} q(\theta).
$
By Danskin's theorem $q^\prime(\theta) = \frac{\partial}{\partial \theta}\Psi(x(\theta),\theta)$, the subgradient of $q$ can be computed as 
\[
q^\prime(\theta) = \ell_2(x(\theta)) - \ell_1(x(\theta)).
\]
Since $q$ is a concave function, the derivative $q^\prime$ is not increasing, and the optimal $\theta^\star$ happens at 
\[
    \theta^\star = \begin{cases}
        0 ,& \text{if } q^\prime(0) \leq  0,\\
        1 ,&\text{if }q^\prime(1) \geq  0,\\
       \text{an interior root of }q^\prime(\theta) = 0 ,& \text{otherwise. }
    \end{cases}
\]
In our implementation, we use bisection to find the root of $q^\prime(\theta) = 0$ in the interval $(0,1)$. We terminate bisection when $|q'(\theta)| \leq 10^{-12} $ or after $80$ iterations. 

\subsection{Reference value $F^\star$ in \cref{subsec:hinge-loss}}
\label{subsec:hinge-loss-Fstar}
Here, we detail how to compute the optimal value of \cref{eq:hinge-loss}.
Recall the identity
$
\max\{0,1-t\}
=
\max_{0\le u\le 1} u(1-t)
$
and $
\sum_{i=1}^m u_i b_i a_i
=
A^\tr(b\odot u)
$
where $A = [a_1^\tr; a_2^\tr; \dots a_m^\tr]$. The problem \cref{eq:hinge-loss} can be equivalently written as
\begin{align}
 &\; \min_x\;
\frac1m\sum_{i=1}^m
\max\{0,1-b_i a_i^\tr x\}
+\frac{\lambda}{2}\|x\|^2 \nonumber \\
= & 
\min_x
\left\{
\frac1m\sum_{i=1}^m
\max_{0\le u_i\le1}
u_i(1-b_i a_i^\tr x)
+
\frac{\lambda}{2}\|x\|^2
\right\}  \nonumber \\
=& 
\min_x
\max_{u\in[0,1]^m}
\left\{
\frac1m\sum_{i=1}^m u_i
-
\frac1m\sum_{i=1}^m
u_i b_i a_i^\tr x
+
\frac{\lambda}{2}\|x\|^2
\right\}  \nonumber \\
=& 
\max_{u\in[0,1]^m} \min_x
\left\{
\frac1m\mathbf 1^\tr u
-
\frac1m
\left\langle A^\tr(b\odot u),x\right\rangle
+
\frac{\lambda}{2}\|x\|^2
\right\}  \nonumber  \\
= & \max_{u\in[0,1]^m}
 q(u):= \left\{
\frac1m\mathbf 1^\tr u - 
\frac{\lambda}{2}
\left\|
\frac{A^\tr(b\odot u)}{m\lambda}
\right\|^2
\right\}, \label{eq:hinge-loss-dual}
\end{align}
where the last equality comes from the optimality condition of the inner minimization
\[
-\frac1m A^\tr(b\odot u)+\lambda x=0 \;\Longleftrightarrow \; x
=
\frac{A^\tr(b\odot u)}{m\lambda}.
\]
In our implementation, we use the algorithm L-BFGS-B in SciPy to solve \cref{eq:hinge-loss-dual} to get $u^\star$ and then use $q(u^\star)$ as a proxy for $F^\star$.
\end{document}

%% file: Preamb.tex
\usepackage{xcolor}
\usepackage{fullpage}
\usepackage{amsmath,amsthm,amssymb,amsfonts}
\usepackage{algorithm} %2e %algorithmic
\usepackage{algpseudocode}
\usepackage[hidelinks]{hyperref}
\hypersetup{
    colorlinks=true,
    linkcolor=blue,
    filecolor=magenta,      
    urlcolor=cyan,
    pdftitle={Revisiting proximal bundle methods},
    pdfpagemode=FullScreen,
}

\usepackage[noblocks]{authblk}
\usepackage{tabularx}%table
\usepackage{booktabs}%table

\usepackage{mathtools}
\usepackage{tcolorbox}
\usepackage{bbm}
\usepackage{todonotes}
\usepackage{tensor}
\usepackage[font=small,labelfont=bf]{caption}
\usepackage{subcaption}
\usepackage{graphicx}
\usepackage{sectsty}

\usepackage{adjustbox} %table
\usepackage{enumitem} %indent

\newcommand{\tr}{\mathsf{ T}}
\newcommand{\RR}{\mathbb{R}}

\newtheorem{theorem}{Theorem}

\newtheorem{remark}{Remark} 
\newtheorem{lemma}{Lemma}
\newtheorem{proposition}{Proposition} 
\newtheorem{corollary}{Corollary} 

\newtheorem{assumption}{Assumption}

\newcommand{\bsf}[1]{\textsf{\LARGE\textbf{#1}}}
\newcommand{\bigline}{\\ \vrule height2pt width 5 in depth 0pt\newline\noindent}

\newcommand{\Dist}{\mathrm{dist}}

\DeclareMathOperator*{\argmin}{\arg\!\min}

\newcommand{\innerproduct}[2]{\left \langle #1, #2 \right\rangle }

\definecolor{moccasin}{rgb}{0.98, 0.92, 0.84}
\newtcolorbox{mybox}{colback=moccasin,
colframe=moccasin}

\usepackage[nameinlink]{cleveref}
\crefname{equation}{}{}
\crefname{theorem}{Theorem}{Theorems}
\crefname{corollary}{Corollary}{Corollaries}
\crefname{example}{Example}{Examples}
\crefname{assumption}{Assumption}{Assumptions}
\crefname{lemma}{Lemma}{Lemmas}
\crefname{proposition}{Proposition}{Propositions}
\crefname{figure}{Figure}{Figures}
\crefname{table}{Table}{Tables}
\crefname{section}{Section}{Sections}
\crefname{appendix}{Appendix}{Appendices}
\Crefname{equation}{}{}
\Crefname{theorem}{Theorem}{Theorems}
\Crefname{corollary}{Corollary}{Corollaries}
\Crefname{example}{Example}{Examples}
\Crefname{lemma}{Lemma}{Lemma}
\Crefname{proposition}{Proposition}{Propositions}
\Crefname{figure}{Figure}{Figures}
\Crefname{table}{Table}{Tables}
\Crefname{section}{Section}{Sections}
\Crefname{appendix}{Appendix}{Appendices}

\newcommand{\bigO}{\mathcal{O}}

\usepackage{wrapfig}
\usepackage{multirow}